\documentclass[10pt,reqno,twoside]{amsart}
\usepackage{amssymb,amsmath,amsthm,soul,color,paralist}
\usepackage{t1enc}
\usepackage{comment}
\usepackage[cp1250]{inputenc}
\usepackage{a4,indentfirst,latexsym}
\usepackage{graphics}
\usepackage{mathrsfs}
\usepackage{cite,enumitem,graphicx}
\usepackage[colorlinks=true,urlcolor=blue,
citecolor=red,linkcolor=blue,linktocpage,pdfpagelabels,
bookmarksnumbered,bookmarksopen]{hyperref}
\usepackage[english]{babel}
\usepackage[left=2.50cm,right=2.50cm,top=2.72cm,bottom=2.72cm]{geometry}
\usepackage[metapost]{mfpic}
\usepackage[hyperpageref]{backref}
\usepackage[colorinlistoftodos]{todonotes}
\usepackage[normalem]{ulem}

\numberwithin{equation}{section}

\def\Xint#1{\mathchoice
  {\XXint\displaystyle\textstyle{#1}}%
  {\XXint\textstyle\scriptstyle{#1}}%
  {\XXint\scriptstyle\scriptscriptstyle{#1}}%
  {\XXint\scriptscriptstyle\scriptscriptstyle{#1}}%
  \!\int}
\def\XXint#1#2#3{{\setbox0=\hbox{$#1{#2#3}{\int}$}
  \vcenter{\hbox{$#2#3$}}\kern-.5\wd0}}
\def\-int{\Xint -}

\newcommand{\R}{\mathbb{R}}
\newcommand{\dist}{\mathit{dist}}

\newcommand{\M}{\mathcal{M}}
\newcommand{\T}{\mathcal{T}}
\newcommand{\h}{H^{s}(\R^{N})}
\newcommand{\X}{X^{1,s}(\R^{N+1}_{+})}
\newcommand{\Xr}{X^{1,s}_{{\rm rad}}(\R^{N+1}_{+})}
\newcommand{\x}{X^{s}(\R^{N+1}_{+})}

\newcommand{\E}{\mathcal{E}}

\newcommand{\2}{2^{*}_{s}}
\newcommand{\wh}{\widehat}
\newcommand{\ri}{\rightarrow}

\DeclareMathOperator{\dive}{div}

\DeclareMathOperator{\e}{\varepsilon}

\newtheorem{prop}{Proposition}[section]
\newtheorem{lem}{Lemma}[section]
\newtheorem{thm}{Theorem}[section]

\newtheorem{cor}{Corollary}[section]

\newtheorem{remark}{Remark}[section]

\title[fractional Kirchhoff equations in $\R^{N}$]{Concentration phenomena for a class of fractional Kirchhoff equations in $\R^{N}$ with  general nonlinearities}

\author[V. Ambrosio]{Vincenzo Ambrosio}
\address{Vincenzo Ambrosio\hfill\break\indent
Dipartimento di Ingegneria Industriale e Scienze Matematiche \hfill\break\indent
Universit\`a Politecnica delle Marche\hfill\break\indent
Via Brecce Bianche, 12\hfill\break\indent
60131 Ancona (Italy)}
\email{v.ambrosio@univpm.it}


\date{}

\begin{document}

\begin{abstract}
In this paper we study the following class of fractional Kirchhoff problems:
\begin{equation*}
\left\{
\begin{array}{ll}
\e^{2s}M(\e^{2s-N}[u]^{2}_{s})(-\Delta)^{s}u + V(x) u= f(u) &\mbox{ in } \R^{N}, \\
u\in H^{s}(\R^{N}), \quad u>0 &\mbox{ in } \R^{N},
\end{array}
\right.
\end{equation*}
where $\e>0$ is a small parameter, $s\in (0, 1)$, $N\geq 2$, $(-\Delta)^{s}$ is the fractional Laplacian, $V:\R^{N}\ri \R$ is a positive continuous function, $M: [0, \infty)\ri \R$ is a Kirchhoff function satisfying suitable conditions and $f:\R\ri \R$ fulfills Berestycki-Lions type assumptions of subcritical or critical type. Using suitable variational arguments, we prove the existence of a family of positive solutions $(u_{\e})$ which concentrates at a local minimum of $V$ as $\e\ri 0$. 
\end{abstract}
\maketitle

\section{Introduction}
\subsection{Main results}
In this paper we deal with the following class of fractional Kirchhoff problems:
\begin{equation}\label{P}
\left\{
\begin{array}{ll}
\e^{2s} M(\e^{2s-N}[u]^{2}_{s})(-\Delta)^{s}u +V(x) u= f(u) &\mbox{ in } \R^{N}, \\
u\in H^{s}(\R^{N}), \quad u>0 &\mbox{ in } \R^{N},
\end{array}
\right.
\end{equation}
where $\e>0$ is a small parameter, $s\in (0, 1)$, $N\geq 2$, 
$M$ is a Kirchhoff function, $V$ is a positive potential and $f$ is a continuous nonlinearity.
The nonlocal operator $(-\Delta)^{s}$ appearing in \eqref{P} is the so called fractional Laplacian operator defined for smooth functions $u: \R^{N}\ri \R$ by
$$
(-\Delta)^{s}u(x)=C(N,s) P.V. \int_{\R^{N}} \frac{u(x)-u(y)}{|x-y|^{N+2s}}\, dy,
$$
where $C(N,s)$ is a positive normalizing constant, 
and $H^{s}(\R^{N})$ denotes the fractional Sobolev space of functions $u\in L^{2}(\R^{N})$ such that
$$
[u]^{2}_{s}:=\iint_{\R^{2N}} \frac{|u(x)-u(y)|^{2}}{|x-y|^{N+2s}}\, dx dy<\infty
$$
endowed with the norm
$$
\|u\|_{\h}:=\sqrt{[u]_{s}^{2}+|u|_{2}^{2}}.
$$
We recall that Fiscella and Valdinoci \cite{FV} 
proposed for the first time a stationary fractional Kirchhoff model in a bounded domain $\Omega\subset \R^{N}$ with homogeneous Dirichlet boundary conditions and involving a critical nonlinearity:
\begin{align}\label{FKE}
\left\{
\begin{array}{ll}
M\left([u]^{2}_{s}\right)(-\Delta)^{s}u=\lambda f(x, u)+|u|^{\2-2}u \quad &\mbox{ in } \Omega,\\
u=0 &\mbox{ in } \R^{N}\setminus \Omega, 
\end{array}
\right. 
\end{align}
where $M$ is a continuous Kirchhoff function whose prototype is given by $M(t)=a+bt$ with $a>0$ and $b\geq 0$, $\lambda>0$ is a parameter and $f$ is a continuous function with subcritical growth.  \\
Their model generalizes in the fractional context the well-known Kirchhoff model introduced by Kirchhoff \cite{K} as an extension of the classical d'Alembert wave equation. For some interesting existence and multiplicity results for Kirchhoff problems in the classic setting, we refer to \cite{ACF, Fig, FIJ, HL, PZ, ZCDO} and the references therein.\\
In the fractional framework, after the pioneering work \cite{FV}, many authors focused on fractional Kirchhoff problems set in bounded domains or in the whole space and involving nonlinearities with subcritical or critical growth; see for instance \cite{AIccm, FP,  MRZ, MBRS, PXZ} and the references therein for unperturbed problems (that is when $\e=1$ in \eqref{P}), and \cite{Aasy, AI2} for some existence and multiplicity results for perturbed problems (that is when $\e>0$ is sufficiently small).

On the other hand, when $M(t)\equiv 1$, equation \eqref{P} boils down to a nonlinear fractional Schr\"odinger equation of the type
\begin{equation}\label{FSE}
\e^{2s}(-\Delta)^{s}u+V(x)u=h(x, u) \mbox{ in } \R^{N},
\end{equation}
proposed by Laskin \cite{Laskin1}  as a result of expanding the Feynman path integral, from the Brownian like to the L\'evy like quantum mechanical paths.
Equation \eqref{FSE} has been object of investigation in these last two decades and several existence and multiplicity results have been obtained under different conditions on $V$ and $h$; see \cite{Adie, Aade, CW, DMV, FQT} and the references therein. In a particular way, a great attention has been devoted to the existence and concentration phenomenon as $\e\ri 0$ of positive solutions to \eqref{FSE}; see \cite{AM, Aampa, DDPW, FigS, He, HZ, Seok, JLZ}. \\
Motivated by the above works, the goal of this paper is to study the existence and  concentration of positive solutions to \eqref{P} under  very general assumptions on the Kirchhoff function $M$ and the nonlinearity $f$. 
We always suppose that $V:\R^{N}\ri \R$ is a continuous function which satisfies the following conditions due to del Pino and Felmer \cite{DF}:
\begin{compactenum}[$(V1)$]
\item $V_{1}:=\inf_{x\in \R^{N}}V(x)>0$,
\item there exists an open bounded set $\Lambda\subset \R^{N}$ such that
$$
V_{0}:=\inf_{x\in \Lambda} V(x)<\min_{x\in \partial \Lambda} V(x).
$$
We also set $\M:=\{x\in \Lambda: V(x)=V_{0}\}$. Without loss of generality, we may assume that $0\in \M$.
\end{compactenum}
Concerning the Kirchhoff function $M$, we suppose that $M:[0, \infty)\ri \R_{+}$ is continuous and such that:
\begin{compactenum}[$(M1)$]
\item there exists $m_{0}>0$ such that $M(t)\geq m_{0}$ for all $t\geq 0$,
\item $\liminf_{t\ri \infty} \left[\wh{M}(t)-(1-\frac{2s}{N})M(t)t\right]=\infty$, where $\wh{M}(t):=\int_{0}^{t} M(\tau)\, d\tau$,
\item $M(t)/t^{\frac{2s}{N-2s}}\ri 0$ as $t\ri \infty$,
\item $M$ is nondecreasing in $[0, \infty)$,
\item $t\mapsto M(t)/t^{\frac{2s}{N-2s}}$ is nonincreasing in $(0, \infty)$.
\end{compactenum}
We note that, if $s=1$, the above assumptions have been used in \cite{FIJ}. Clearly, 
$M(t)=m_{0}+bt$, with $ b\geq 0$, satisfies $(M1)$-$(M5)$  when $b=0$, $N\geq 2$ and $s\in (0, 1)$, and $N=3$, $s\in (\frac{3}{4}, 1)$ whenever $b>0$.\\
In the first part of the paper, we require that $f:\R\ri \R$ is a continuous function such that $f(t)=0$ for $t\leq 0$ and fulfills the following Beresticky-Lions type assumptions \cite{BL}:
\begin{compactenum}[$(f_1)$]
\item $\lim_{t\ri 0} \frac{f(t)}{t}=0$,
\item $\limsup_{t\ri \infty} \frac{f(t)}{t^{p}}<\infty$ for some $p\in (1, \2-1)$, where $\2:=\frac{2N}{N-2s}$ is the fractional critical exponent,
\item there exists $T>0$ such that $F(T)>\frac{V_{0}}{2}T^{2}$, where $F(t):=\int_{0}^{t} F(\tau)\, d\tau$.
\end{compactenum}
The first main result  of this work can be stated as follows:
\begin{thm}\label{thm1}
Assume that $(V_1)$-$(V_2)$, $(M_1)$-$(M_5)$ and $(f_1)$-$(f_3)$ are satisfied. When $s\in (0, \frac{1}{2}]$, we also assume that $f\in C^{0, \alpha}_{loc}(\R)$ for some $\alpha\in (1-2s, 1)$.
Then, for small $\e>0$, there exists a positive solution $u_{\e}$ to \eqref{P}. Moreover, there exists a maximum point $x_{\e}\in \R^{N}$ of $u_{\e}$ such that $\lim_{\e\ri 0} dist(x_{\e}, \M)=0$, and for any such $x_{\e}$, $v_{\e}(x)=u_{\e}(\e x+x_{\e})$ converges, up to a subsequence, in $H^{s}(\R^{N})$ to a least energy solution of the limiting problem 
$$
M([u]^{2}_{s})(-\Delta)^{s}u+V_{0} u=f(u) \mbox{ in } \R^{N}.
$$
In particular, there exists a constant $C>0$, independent of $\e>0$, such that
$$
u_{\e}(x)\leq \frac{C\e^{N+2s}}{\e^{N+2s}+|x-x_{\e}|^{N+2s}} \quad \forall x\in \R^{N}.
$$
\end{thm}
\begin{remark}
The restrictions on the regularity on $f$ are only used to obtain the better regularity of solutions to \eqref{P} which guarantees the Pohozaev identity (see Proposition $1.1$ in \cite{BKS}). 
\end{remark}

In the second part of this paper, we consider \eqref{P} by requiring that $f$
satisfies the following Beresticky-Lions type assumptions of critical growth \cite{ZZ}, that is $f$ fulfills $(f_1)$ and
\begin{compactenum}[]
\item $(f'_2)$ $\lim_{t\ri \infty} \frac{f(t)}{t^{\2-1}}=1$,
\item $(f'_3)$ there exist $\lambda>0$ and $p<\2$ such that
$$
f(t)\geq t^{\2-1}+\lambda t^{p-1} \quad \forall t\geq 0,
$$
where $\lambda>0$ is such that
\begin{itemize}
\item $p\in (2, \2)$ and $\lambda>0$ if $N\geq 4s$,
\item $p\in (\frac{4s}{N-2s}, \2)$ and $\lambda>0$ if $2s<N<4s$,
\item $p\in (2, \frac{4s}{N-2s}]$ and $\lambda>0$ is sufficiently large if $2s<N<4s$.
\end{itemize}
\end{compactenum}
Then, the second main result of this paper is the following:
\begin{thm}\label{thm2}
Assume that $(V_1)$-$(V_2)$, $(M_1)$-$(M_5)$ and $(f_1)$, $(f'_2)$-$(f'_3)$ are satisfied. When $s\in (0, \frac{1}{2}]$, we also assume that $f\in C^{0, \alpha}_{loc}(\R)$ for some $\alpha\in (1-2s, 1)$.
Then, for small $\e>0$, there exists a positive solution $u_{\e}$ to \eqref{P}. Moreover, there exists a maximum point $x_{\e}\in \R^{N}$ of $u_{\e}$ such that $\lim_{\e\ri 0} dist(x_{\e}, \M)=0$, and for any such $x_{\e}$, $v_{\e}(x)=u_{\e}(\e x+x_{\e})$ converges, up to a subsequence, in $H^{s}(\R^{N})$ to a least energy solution of  
$$
M([u]^{2}_{s})(-\Delta)^{s}u+V_{0} u=f(u) \mbox{ in } \R^{N}.
$$
\end{thm}

\subsection{State of the art and methodology}
We point out that Theorem \ref{thm1} and Theorem \ref{thm2} can be seen as the nonlocal fractional counterpart of Theorem $1.1$ in \cite{FIJ} and Theorem $1.1$ in \cite{ZCDO}, respectively. We recall that in \cite{FIJ} Figueiredo et al. refined some arguments developed in \cite{BJ,  BJT, BW}, in which the authors 
studied the existence and concentration of positive solutions {\color{red}{for}} the nonlinear Schr\"odinger equation
\begin{equation}\label{NSE}
-\e^{2}\Delta u+V(x)u=f(u) \mbox{ in } \R^{N},
\end{equation}
and involving general subcritical nonlinearities.
More precisely, Byeon and Jeanjean \cite{BJ} explored what are the essential features on $f$ which guarantee the existence of localized ground states. To do this, the authors developed a new variational approach which consists in searching solutions of \eqref{NSE} in a neighborhood of the set of the least energy solution of the limiting problem associated with \eqref{NSE} whose mass stays close to $\M$; see \cite{BJ2, BJT, BW} for more details. 
Subsequently, motivated by \cite{FIJ, ZZ}, Zhang et al. \cite{ZCDO} extended the result in \cite{FIJ} when $f$ is a general critical nonlinearity by applying a suitable truncation argument. 

The purpose of this work is to generalize the results in \cite{FIJ, ZCDO} to the fractional setting $s\in (0, 1)$.\\
For the sake of completeness, we start to mention some recent results in the case $M(t)\equiv 1$, that is when \eqref{P} reduces to the fractional Schr\"odinger equation \eqref{FSE}. Seok \cite{Seok} proved the existence of multi-peak solutions to \eqref{FSE} by assuming $(f_1)$-$(f_3)$ and extending in the nonlocal framework the result in \cite{BJ2}. In \cite{Seok}, the author  did not introduce a penalization term as in \cite{BJ, BJ2} but proved a kind of intersection lemma by using degree theory after transforming \eqref{FSE} into a degenerate elliptic problem via the extension method \cite{CS}. In \cite{JLZ} Jin et al. considered \eqref{FSE} under conditions $(f_1)$, $(f'_2)$-$(f'_3)$ and constructed a family of positive solutions to \eqref{FSE} which concentrates at a local minimum of $V$ as $\e\ri 0$.
The authors combined the extension method, a truncation argument inspired by \cite{ZCDO} with the result in \cite{Seok}. 
Simultaneously, He \cite{He} obtained the same result by applying the extension method and combining the penalization methods developed in \cite{BW} and \cite{DF}, respectively. We stress that this last approach has been previously used by Gloss \cite{Gloss} to extend the result in \cite{BJ} to a $p$-Laplacian problem involving a general subcritical nonlinearity. \\
We note that the results in \cite{He, JLZ, Seok} improve the previous ones obtained in \cite{AM, Aampa, HZ} in which the authors, motivated by \cite{DF}, considered nonlinearities satisfying the Ambrosetti-Rabinowitz condition \cite{AR} and by requiring that $\frac{f(t)}{t}$ is strictly increasing for $t>0$. Indeed, under assumptions $(f_1)$-$(f_3)$ or $(f_1)$, $(f'_2)$-$(f'_3)$, the Nehari method developed in the above mentioned papers does not work and it is very hard to verify the Palais-Smale  compactness condition in this situation; see \cite{Armi} for more details.\\
Concerning fractional Kirchhoff problems, to our knowledge, only few papers deal with the existence and concentration behavior of positive solutions  as $\e\ri 0$. In fact, motivated by \cite{AM, Aampa, HZ}, 
in \cite{Aasy, AI2, HZm} the authors 
studied the existence and concentration phenomena to \eqref{P} when $M(t)=a+bt$, $N=3$ and $s\in (\frac{3}{4},1)$.
However, the nonlinearities in \cite{Aasy, AI2, HZm} are less general than the ones presented here. \\
In this paper, by using suitable variational methods, we improve the results in \cite{Aasy, AI2, HZm} by considering a more general class of fractional Kirchhoff problems in the whole space $\R^{N}$, with $N\geq 2$. 
More precisely, after realizing \eqref{P} as a local linear degenerate
elliptic equation in $\R^{N+1}_{+}$ together with a nonlinear Neumann boundary condition on $\partial \R^{N+1}_{+}$,
we take inspiration by the penalization approach in \cite{BJ, DF, Gloss} and some arguments used in \cite{AM, Aasy, AI2, FIJ, He, JLZ, ZCDO}, to obtain the existence of a family of positive solutions which concentrates around a local minimum of the potential $V(x)$, as $\e\rightarrow 0$.
We emphasized that, making use of the extension method, several techniques used in the case $s=1$ cannot be directly adapted in our setting because we have to take care of the traces terms of the involved functions and to work with weighted Lebesgue spaces. Moreover, due to the presence of the Kirchhoff term, our analysis is much more delicate and intriguing with respect to the case $M(t)\equiv 1$ and $s\in (0, 1)$ discussed above. For instance, if $(u_{\e})$ is a bounded sequence in $\h$ of solutions to \eqref{P} such that $u_{\e}(\e x+x_{\e})\rightharpoonup u$ in $\h$ and $x_{\e}\ri x_{0}$ as $\e\ri 0$, then $u$ is solution to the limiting problem $\alpha_{0}(-\Delta)^{s}u+V(x_{0})u=f(u)$ in $\R^{N}$, where $\alpha_{0}:=\lim_{\e\ri 0} M([u_{\e}]_{s}^{2})$, and in general it is complicated to verify that $\alpha_{0}=M([u]^{2}_{s})$. Therefore, some refined estimates will be needed to overcome these difficulties; see Lemma \ref{lem4.1G} and Lemma \ref{Klem}.\\
As far as we know, these are the first existence results for \eqref{P} under local assumptions on the potential $V$ and general nonlinearities $f$ with subcritical or critical growth.

\smallskip
\noindent
The paper is organized as follows. In section $2$ we introduce the notations and we recall some useful results. In section $3$ we study the limiting Kirchhoff problem associated with \eqref{P} by assuming $(f_1)$-$(f_3)$. The critical limiting Kirchhoff problem is considered in section $4$. In section $5$ we provide the proof of Theorem \ref{thm1}. The last section is devoted to the proof of Theorem \ref{thm2}.


\section{preliminaries}
In this section we fix the notations and collect some preliminary results for future references.
For more details we refer to \cite{CSire, CS, DPV, DMV, MBRS}.

We denote the upper half-space in $\R^{N+1}$ by
$$
\R^{N+1}_{+}:=\{(x, y)\in \R^{N+1}: y>0\}.
$$
For $p\in [1, \infty]$, let $L^{p}(\R^{N})$ be the set of measurable functions $u: \R^{N}\ri \R$ such that
\begin{equation*}
|u|_{p}:=\left\{
\begin{array}{ll}
\left(\int_{\R^{N}} |u|^{p}\, dx\right)^{1/p}<\infty &\mbox{ if } p<\infty, \\
{\rm esssup}_{x\in \R^{N}} |u(x)|  &\mbox{ if } p=\infty.
\end{array}
\right.
\end{equation*}
Let $\mathcal{D}^{s, 2}(\R^{N})$, with $s\in (0, 1)$, be the completion of $C^{\infty}_{c}(\R^{N})$ with respect to the Gagliardo seminorm
$$
[u]_{s}:=\left( \iint_{\R^{2N}} \frac{|u(x)-u(y)|^{2}}{|x-y|^{N+2s}}\, dx dy \right)^{\frac{1}{2}}.
$$
Then (see \cite{DPV}) the embedding $\mathcal{D}^{s, 2}(\R^{N})\subset L^{\2}(\R^{N})$ is continuous and 
$$
|u|_{\2}\leq c(N,s) [u]_{s} \quad \forall u\in \mathcal{D}^{s, 2}(\R^{N}).
$$
Denote by $H^{s}(\R^{N})$ the fractional Sobolev space
$$
H^{s}(\R^{N}):=\{u\in L^{2}(\R^{N}): [u]_{s}<\infty\}
$$
endowed with the norm
$$
\|u\|_{\h}:=\left([u]_{s}^{2}+|u|^{2}_{2}\right)^{\frac{1}{2}}.
$$
Then, $H^{s}(\R^{N})$ is continuously embedded in $L^{p}(\R^{N})$ for all $p\in [2, \2)$ and compactly in  $L^{p}_{loc}(\R^{N})$ for all $p\in [1, \2)$; see \cite{DPV}. We also define the fractional radial Sobolev space
$$
H^{s}_{\rm rad}(\R^{N}):=\{u\in \h: u(x)=u(|x|)\}.
$$
It is well-known (see \cite{Lions}) that $H^{s}_{\rm rad}(\R^{N})$ is compactly embedded in $L^{q}(\R^{N})$ for all $q\in (2, \2)$.

Let us define $X^{s}(\R^{N+1}_{+})$ as the completion of $C^{\infty}_{c}(\overline{\R^{N+1}_{+}})$ under the norm
$$
\|u\|_{X^{s}(\R^{N+1}_{+})}:=  \left(\iint_{\R^{N+1}_{+}} y^{1-2s} |\nabla u|^{2}\, dx dy  \right)^{\frac{1}{2}}.
$$
Then (see \cite{BCDPS}) there exists a linear trace operator ${\rm Tr}: \x\ri \mathcal{D}^{s,2}(\R^{N})$ such that 
$$
\sqrt{\kappa_{s}} [{\rm Tr}(u)]_{s}\leq \|u\|_{\x} \mbox{ for any } u\in \x,
$$ 
where $\kappa_{s}:=2^{1-2s}\Gamma(1-s)/\Gamma(s)$. In what follows, we set $u(\cdot, 0):={\rm Tr}(u)$.\\
Denote by 
$$
B^{+}_{R}(x_{0}, y_{0}):=\{(x, y)\in \R^{N+1}_{+}: |(x,y)-(x_{0}, y_{0})|<R\}
$$ 
the open ball in $\R^{N+1}_{+}$ with center $(x_0, y_0)\in \R^{N+1}_{+}$ and radius $R>0$, and 
$$
\Gamma_{R}^{0}(z_{0}):=\{(x, 0)\in \partial \R^{N+1}_{+}: |x-z_{0}|<R\}
$$ 
the ball in $\R^{N}$ with center $z_{0}\in \R^{N}$ and radius $R>0$.\\
We denote by $X^{s}_{0}(B_{R}^{+}(0,0))$, with $R>0$,  the completion of $C^{\infty}_{c}(B_{R}^{+}(0,0)\cup \Gamma^{0}_{R}(0))$ under the norm
$$
\|u\|_{X^{s}_{0}(B_{R}^{+}(0,0))}:=  \left(\iint_{B_{R}^{+}(0,0)} y^{1-2s} |\nabla u|^{2}\, dx dy  \right)^{\frac{1}{2}}.
$$
Note that if $w\in X^{s}_{0}(B_{R}^{+}(0,0))$ then its extension by zero outside $B_{R}^{+}(0,0)$ can be approximated by functions with compact support in $\overline{\R^{N+1}_{+}}$.
Moreover, for all $r\in [1, \2]$ and $u\in X^{s}_{0}(B_{R}^{+}(0,0))$ it holds (see \cite{BCDPS})
$$
C(r,s,N,R)\left(\int_{\Gamma_{R}^{0}(0)} |u(\cdot, 0)|^{r}\, dx\right)^{\frac{2}{r}}\leq \iint_{B_{R}^{+}(0,0)} y^{1-2s}|\nabla u|^{2}\, dxdy.
$$
We define 
$$
X^{1,s}(\R^{N+1}_{+}):=\left\{u\in X^{s}(\R^{N+1}_{+}): \int_{\R^{N}} u^{2}(x,0)\, dx<\infty  \right\}
$$
equipped with the norm
$$
\|u\|_{X^{1,s}(\R^{N+1}_{+})}:=\left(\iint_{\R^{N+1}_{+}} y^{1-2s} |\nabla u|^{2}\, dx dy+\int_{\R^{N}} u^{2}(x,0)\, dx   \right)^{\frac{1}{2}}.
$$
Finally, we consider
$$
X^{1,s}_{\rm rad}(\R^{N+1}_{+}):=\{u\in \X: u(x,y)=u(|x|,y)\}.
$$
\noindent
The following Sobolev inequality holds true:
\begin{lem}\cite{BCDPS}\label{Sobolev}
For every $u\in \X$ it holds for some positive constant $S(s, N)>0$
$$
S(s,N) \left(\int_{\R^{N}} |u(x, 0)|^{\2}\, dx\right)^{\frac{2}{\2}}\leq \iint_{\R^{N+1}_{+}} y^{1-2s} |\nabla u|^{2}\, dx dy.
$$
\end{lem}

\noindent
For all $r\in (1, \infty)$, we define the weighted Lebesgue space $L^{r}(\R^{N+1}_{+}, y^{1-2s})$ endowed with the norm
$$
\iint_{\R^{N+1}_{+}} y^{1-2s} |u|^{r}\, dx dy.
$$
We recall the following useful result proved in \cite{DMV}:
\begin{lem}\label{lem2.1}\cite{DMV}
\begin{compactenum}[$(i)$]
\item There exists a constant $C>0$ such that for all $w\in X^{s}(\R^{N+1}_{+})$ it holds
$$
\left( \iint_{\R^{N+1}_{+}} y^{1-2s} |w|^{2\gamma}\, dx dy \right)^{\frac{1}{2\gamma}}\leq C\left(\iint_{\R^{N+1}_{+}} y^{1-2s} |\nabla w|^{2} \, dx dy \right)^{\frac{1}{2}},
$$
where $\gamma:=1+\frac{2}{N-2s}$.
\item Let $R>0$ and $\mathcal{T}$ be a subset of $X^{s}(\R^{N+1}_{+})$ such that
$$
\sup_{u\in \mathcal{T}} \int_{\R^{N+1}_{+}} y^{1-2s} |\nabla W|^{2}\, dx dy<\infty.
$$
Then, $\mathcal{T}$ is compact in $L^{2}(B_{R}^{+}(0,0), y^{1-2s})$.
\end{compactenum}
\end{lem}

The fractional Laplacian $(-\Delta)^{s}$ may be defined for $u:\R^{N}\ri \R$  belonging to the Schwartz space of rapidly decaying functions by
$$
(-\Delta)^{s}u(x)=C(N,s) P. V. \int_{\R^{N}} \frac{u(x)-u(y)}{|x-y|^{N+2s}}\, dy
$$
where
$$
C(N,s):=\left(\int_{\R^{N}} \frac{1-\cos(x_1)}{|x|^{N+2s}}\, dx\right)^{\frac{1}{2}}.
$$
It can be also defined using Fourier transform by
$$
\mathcal{F}((-\Delta)^{s}u(k))=|k|^{2s}\mathcal{F}u(k). 
$$
It is well-known (see \cite{DPV}) that for all $u\in \h$
$$
|(-\Delta)^{\frac{s}{2}}u|_{2}^{2}=\int_{\R^{N}}|k|^{2s} |\mathcal{F}u(k)|^{2}\, dk=\frac{1}{2} C(N,s) [u]^{2}_{s}.
$$
In \cite{CS}, it is showed that one can see $(-\Delta)^{s}$ by considering it as the Dirichlet to Neumann operator associated to the $s$-harmonic extension in the half-space, paying the price to add a new variable.
More precisely, for any $u\in \mathcal{D}^{s,2}(\R^{N})$ there exists a unique function $U\in \x$ solving the following problem
\begin{align*}
\left\{
\begin{array}{ll}
-\dive(y^{1-2s} \nabla U)=0 &\mbox{ in } \R^{N+1}_{+}, \\
U(\cdot, 0)=u &\mbox{ on } \partial \R^{N+1}_{+}=\R^{N}. 
\end{array}
\right.
\end{align*}
The function $U$ is called the $s$-harmonic extension of $u$ and possesses the following properties:
\begin{compactenum}[$(i)$]
\item
$$
\frac{\partial U}{\partial \nu^{1-2s}}:=-\lim_{y\ri 0} y^{1-2s} \frac{\partial U}{\partial y}(x,y)=\kappa_{s}(-\Delta)^{s}u(x) \mbox{ in distribution sense, }
$$
\item $\sqrt{\kappa_{s}}[u]_{s}=\|U\|_{\x}\leq \|V\|_{\x}$ for all $V\in \x$ such that $V(\cdot,0)=u$.
\item $U\in C^{\infty}(\R^{N+1}_{+})\cap L^{2}(K, y^{1-2s})$ for any compact set $K\subset \overline{\R^{N+1}_{+}}$,
$$
U(x, y)=\int_{\R^{N}} P_{s}(x-z,y) u(z)\, dz
$$
where
$$
P_{s}(x,y):=p_{N,s} \frac{y^{2s}}{(|x|^{2}+y^{2})^{\frac{N+2s}{2}}}
$$
and $p_{N,s}$ is a positive constant such that $\int_{\R^{N}} P_{s}(x,y)\, dx=1$ for all $y>0$.
\end{compactenum}

\noindent
Using the change of variable $x\mapsto \e x$, it is possible to prove that \eqref{P} is equivalent to the following problem
\begin{equation}\label{SP}
\left\{
\begin{array}{ll}
M([u]^{2}_{s})(-\Delta)^{s}u + V_{\e}(x) u= f(u) &\mbox{ in } \R^{N}, \\
u\in H^{s}(\R^{N}), \quad u>0 &\mbox{ in } \R^{N},
\end{array}
\right.
\end{equation}
where $V_{\e}(x):=V(\e x)$.
Then, in view of the previous facts, problem \eqref{SP} can be realized in a local manner through the nonlinear boundary value problem:
\begin{align}\label{EP}
\left\{
\begin{array}{ll}
-\dive(y^{1-2s} \nabla w)=0 &\mbox{ in } \R^{N+1}_{+}, \\
\frac{1}{M(\|w\|^{2}_{\x})}\frac{\partial w}{\partial \nu^{1-2s}}=\kappa_{s} [-V_{\e} w(\cdot, 0)+f(w(\cdot, 0))] &\mbox{ in } \R^{N}. 
\end{array}
\right.
\end{align}
For simplicity we will drop the constant $\kappa_{s}$ from the second equation in \eqref{EP}.

\section{Subcritical limiting problems}
We begin by modifying $f$ {\color{red}{as in}} \cite{BL}. Let $\hat{f}: \R\ri \R$ be defined as follows:
\begin{compactenum}[$(i)$]
\item if $f(t)>0$ for all $t\geq \wh{T}$, put $\hat{f}(t):=f(t)$,
\item if there exists $\tau_{0}\geq \wh{T}$ such that $f(\tau_{0})=0$, we put
\begin{equation*}
\hat{f}(t):=\left\{
\begin{array}{ll}
f(t) &\mbox{ for } t<\tau_{0}, \\
0  &\mbox{ for } t\geq \tau_{0},
\end{array}
\right.
\end{equation*}
where $\wh{T}:=\sup\{ t\in [0, T]: f(t)>V_{0}t \}$.
\end{compactenum}
Note that $\hat{f}$ satisfies the same assumptions as $f$ and
$$
0\leq \liminf_{t\ri \infty} \frac{\hat{f}(t)}{t^{p}}\leq \limsup_{t\ri \infty} \frac{\hat{f}(t)}{t^{p}}<\infty.
$$
Moreover, if $(ii)$ occurs and $u$ is a solution to \eqref{P} with $\hat{f}(t)$, then we can use $(u-\tau_{0})_{+}$ as test function to deduce that $u\leq \tau_{0}$ in $\R^{N}$, that is $u$ is a solution to \eqref{P} with $f(t)$. From now on, we replace $f$ by $\hat{f}$ and  keep the same notation $f(t)$.

\noindent
In this section we focus on the following limiting problem associated with \eqref{EP}:
\begin{align}\label{LP}
\left\{
\begin{array}{ll}
-\dive(y^{1-2s} \nabla w)=0 &\mbox{ in } \R^{N+1}_{+}, \\
\frac{1}{M(\|w\|^{2}_{\x})}\frac{\partial w}{\partial \nu^{1-2s}}=-V_{0} w(\cdot, 0)+f(w(\cdot, 0)) &\mbox{ in } \R^{N}. 
\end{array}
\right.
\end{align}
To obtain our results we take inspiration by some arguments used in \cite{FIJ, HL}.
Firstly, we show that the solutions of \eqref{LP} satisfy a Pohozaev identity.
\begin{lem}\label{lem2.11FIJ}
Assume that $(M1)$ holds and $u\in \X$ is a solution to \eqref{LP}. Then $u$ satisfies the following Pohozaev type identity: 
\begin{align*}
P(u) := \frac{N-2s}{2} M( \|u\|^{2}_{\x} ) \|u\|^{2}_{\x} - N \int_{\R^{N}} F(u(x,0)) - \frac{V_{0}}{2} u^{2}(x,0)\, dx =0. 
\end{align*} 
\end{lem}
\begin{proof}
Put $\alpha_{0}:=M( \|u\|^{2}_{\x})$. Then $u$ is a solution to 
\begin{align*}
\left\{
\begin{array}{ll}
-\dive(y^{1-2s} \nabla u)=0 &\mbox{ in } \R^{N+1}_{+}, \\
\frac{1}{\alpha_{0}}\frac{\partial u}{\partial \nu^{1-2s}}=-V_{0} u(\cdot, 0)+f(u(\cdot, 0)) &\mbox{ in } \R^{N}. 
\end{array}
\right.
\end{align*}
Arguing as in \cite{Aade, Adie, BKS, CW}, we deduce that $u$ satisfies the following Pohozaev identity
$$
\frac{N-2s}{2} \alpha_{0} \|u\|^{2}_{\x} - N \int_{\R^{N}} F(u(x,0)) - \frac{V_{0}}{2} u^{2}(x,0)\, dx =0
$$
which implies the thesis.
\end{proof}

In order to find weak solutions to \eqref{LP}, we look for critical points of the energy functional $L_{V_{0}}: \X\ri \R$ defined as
$$
L_{V_{0}}(u):=\frac{1}{2} \wh{M}\left(\|u\|^{2}_{X^{s}(\R^{N+1}_{+})}\right)+\frac{1}{2}\int_{\R^{N}} V_{0} u^{2}(x,0)\, dx-\int_{\R^{N}} F(u(x,0))\, dx.
$$
From $(f_1)$-$(f_2)$, it is easy to check that $L_{V_{0}}\in C^{1}(\X, \R)$. Moreover, we see that $L_{V_{0}}$ possesses a nice 
geometric structure. 
\begin{lem}\label{MPTL}
Assume $(M1)$-$(M3)$. Then, $L_{V_{0}}$ has a mountain pass geometry.
\end{lem}
\begin{proof}
By $(M1)$, $(f_1)$, $(f_2)$ and $\h\subset L^{p+1}(\R^{N})$ we have
\begin{align*}
L_{V_{0}}(u)&\geq \frac{m_{0}}{2} \|u\|^{2}_{X^{s}(\R^{N+1}_{+})}+\frac{V_{0}}{2}|u(\cdot, 0)|^{2}_{2}-\e |u(\cdot, 0)|_{2}^{2}-C_{\e} |u(\cdot, 0)|_{p+1}^{p+1} \\
&\geq c_{1} \|u\|_{\X}^{2}-c_{2}  \|u\|_{\X}^{p+1}.
\end{align*} 
Hence, there exist $\rho, \delta>0$ such that $L_{V_{0}}(u)\geq \delta$ for $\|u\|_{\X}=\rho$.\\
Now, for all $R>0$ we define 
\begin{equation*}
w_{R}(x,y):=\left\{
\begin{array}{ll}
T &\mbox{ if } (x,y)\in B^{+}_{R}(0,0), \\
T\left(R+1-\sqrt{|x|^{2}+y^{2}} \right)  &\mbox{ if } (x,y)\in B^{+}_{R+1}(0, 0)\setminus B^{+}_{R}(0, 0), \\
0 &\mbox{ if } (x, y)\in \R^{N+1}_{+}\setminus B_{R+1}^{+}(0,0).
\end{array}
\right.
\end{equation*}
It is clear that $w_{R}\in \Xr$. Note that, by $(f_3)$, for $R>0$ large enough it holds
$$
\int_{\R^{N}} F(w_{R}(x,0))-\frac{V_{0}}{2}w^{2}_{R}(x,0)\, dx\geq 1.
$$
Now, fix such an $R>0$ and consider $w_{R, \theta}(x,y):=w_{R}(x/e^{\theta}, y/e^{\theta})$. Then,
\begin{align*}
L_{V_{0}}(w_{R, \theta})&=\frac{1}{2} \wh{M}(e^{(N-2s)\theta} \|w_{R}\|^{2}_{X^{s}(\R^{N+1}_{+})})-e^{N\theta} \int_{\R^{N}} F(w_{R}(x,0))-\frac{V_{0}}{2}w^{2}_{R}(x,0)\, dx\\
&\leq \frac{1}{2} \wh{M}(e^{(N-2s)\theta}  \|w_{R}\|^{2}_{X^{s}(\R^{N+1}_{+})})-e^{N\theta} \ri -\infty \mbox{ as } \theta\ri \infty
\end{align*}
because $(M3)$ yields 
$$
e^{-N\theta} \wh{M}(e^{(N-2s)\theta} \|w_{R}\|^{2}_{\x})\ri 0 \mbox{ as } \theta\ri \infty.
$$
\end{proof}

\noindent
In view of Lemma \ref{MPTL} we can define the minimax level
\begin{equation}\label{cv0}
c_{V_{0}}:=\inf_{\gamma\in \Gamma_{V_{0}}} \max_{t\in [0,1]} L_{V_{0}}(\gamma(t))
\end{equation}
and
\begin{equation}\label{gammav0}
\Gamma_{V_{0}}:=\{\gamma\in C([0, 1], \X): \gamma(0)=0, L_{V_{0}}(\gamma(1))<0\}.
\end{equation} 
Obviously, $c_{V_{0}}>0$. We can also note that 
\begin{equation}\label{cmp=cmpr}
c_{V_{0}}=c_{V_{0}, \rm{rad}},
\end{equation}
where
$$
c_{V_{0}, \rm{rad}}:=\inf_{\gamma\in \Gamma_{V_{0}, rad}} \max_{t\in [0, 1]} L_{V_{0}}(\gamma(t)),
$$
and
$$
\Gamma_{V_{0}, \rm{rad}}:=\{\gamma\in C([0, 1], \Xr): \gamma(0)=0, L_{V_{0}}(\gamma(1))<0\}.
$$
Indeed, $c_{V_{0}}\leq c_{V_{0}, \rm{rad}}$ by the definitions. 
For the opposite inequality, take $\gamma\in \Gamma_{V_{0}}$ and consider $\gamma_{\e}(t):=\rho_{\e}*\gamma(t)$, where $\rho_{\e}\in C^{\infty}_{c}(\R^{N+1}_{+})$ is a standard mollifier. Then, $\gamma_{\e}\in C([0, 1], \X)$, $\gamma_{\e}(0)=0$ and $\gamma_{\e}(t)\in C^{\infty}(\R^{N+1}_{+})\cap \X$ for all $t\in [0, 1]$. 
Since 
$$
\sup_{t\in [0, 1]} \|\gamma_{\e}(t)-\gamma(t)\|_{\X}\ri 0 \mbox{ as } \e\ri 0,
$$
we deduce that
$$
\max_{t\in [0, 1]} L_{V_{0}}(\gamma_{\e}(t))\ri \max_{t\in [0, 1]} L_{V_{0}}(\gamma(t)) \mbox{ as } \e\ri 0.
$$ 
Now, let $\phi_{\e}^{*}(t)$ be the symmetric decreasing rearrangement of $\gamma_{\e}(t)(\cdot, 0)\in H^{s}(\R^{N})$, and denote by $\gamma_{\e}^{*}(t)$ the solution of 
\begin{align*}
\left\{
\begin{array}{ll}
-\dive(y^{1-2s} \nabla \gamma_{\e}^{*}(t))=0 &\mbox{ in } \R^{N+1}_{+}, \\
\gamma_{\e}(t)(\cdot, 0)=\phi_{\e}^{*}(t) &\mbox{ in } \R^{N}. 
\end{array}
\right.
\end{align*}
Since $\gamma_{\e}^{*}(t)$ is the $s$-harmonic extension of $\phi_{\e}^{*}(t)$, and using the trace inequality and Theorem 9.2 in \cite{AL} we have
$$
\|\gamma_{\e}^{*}(t)\|_{\x}=[\phi_{\e}^{*}(t)]_{s}\leq [\gamma_{\e}(t)(\cdot, 0)]_{s}\leq \|\gamma_{\e}(t)\|_{\x}.
$$
On the other hand, for all $G:\R\ri \R$ continuous
$$
\int_{\R^{N}} G(\gamma_{\e}^{*}(t)(\cdot, 0))\, dx=\int_{\R^{N}} G(\phi_{\e}^{*}(t))\, dx=\int_{\R^{N}} G(\gamma_{\e}(t)(\cdot, 0))\, dx.
$$
Observing that $\wh{M}$ is strictly increasing (by $(M1)$), we obtain that $L_{V_{0}}(\gamma^{*}_{\e}(t))\leq L_{V_{0}}(\gamma_{\e}(t))$ for all $t\in [0, 1]$. Moreover, since $\gamma_{\e}(\cdot, 0)\in C^{\infty}(\R^{N})$, we have that $\gamma_{\e}(\cdot, 0)$ is co-area regular (see \cite{AL}) and using Theorem 9.2 in \cite{AL} we deduce that $\phi^{*}_{\e}\in C([0, 1], H^{s}_{\rm rad}(\R^{N}))$ and consequently $\gamma_{\e}^{*}\in C([0, 1], \Xr)$. In conclusion, $\gamma^{*}_{\e}\in \Gamma_{V_{0}, \rm{rad}}$ and \eqref{cmp=cmpr} holds true.

Now we prove the existence of a Palais-Smale sequence of $L_{V_{0}}$ with an extra property related to the Pohozaev identity; see \cite{FIJ, HIT, HL}.

\begin{prop}\label{prop3.4HLP}
There exists a sequence $(w_{n})\subset \Xr$ such that
\begin{align}\label{3.6HLP}
L_{V_{0}}(w_{n})\ri c_{V_{0}},\, L'_{V_{0}}(w_{n})\ri 0, \, P(w_{n})\ri 0.
\end{align}
\end{prop}
\begin{proof}
Let $\widetilde{L}_{V_{0}}(\theta, u):=(L_{V_{0}}\circ \Phi)(\theta, u)$ for $(\theta, u)\in \R\times \Xr$, where $\Phi(\theta, u):=u(\frac{x}{e^{\theta}}, \frac{y}{e^{\theta}})$. 
Here $\R\times \Xr$ is equipped with the standard norm
$$
\|(\theta, u)\|_{\R\times \X}:=(|\theta|^{2}+\|u\|^{2}_{\X})^{\frac{1}{2}}.
$$
It follows from Lemma \ref{MPTL} that $\widetilde{L}_{V_{0}}$ has a mountain pass geometry, so we can define the mountain pass level of  $\widetilde{L}_{V_{0}}$
$$
\tilde{c}_{V_{0}}:=\inf_{\widetilde{\gamma}\in \widetilde{\Gamma}_{V_{0}}} \max_{t\in [0,1]} \widetilde{L}_{V_{0}}(\widetilde{\gamma}(t))
$$
where
$$
\widetilde{\Gamma}_{V_{0}}:=\{\widetilde{\gamma}\in C([0, 1], \R\times \Xr): \widetilde{\gamma}(0)=(0,0), \widetilde{L}_{V_{0}}(\widetilde{\gamma}(1))<0\}.
$$ 
It is easy to show that $\widetilde{c}_{V_{0}}=c_{V_{0}}$ (see \cite{Aade, HIT}). Then, by the general minimax principle (see Theorem 2.8 in \cite{W}), we deduce that there exists a sequence $((\theta_{n}, u_{n}))\subset \R\times \Xr$ such that, as $n\ri \infty$,
\begin{compactenum}[$(i)$]
\item $(L_{V_{0}}\circ \Phi)(\theta_{n}, u_{n})\ri c_{V_{0}}$,
\item $(L_{V_{0}}\circ \Phi)'(\theta_{n}, u_{n})\ri 0$ in $(\R\times \Xr)'$,
\item $\theta_{n}\ri 0$.
\end{compactenum}
Indeed, if we take $\e=\e_{n}=\frac{1}{n^{2}}$, $\delta=\delta_{n}=\frac{1}{n}$ in Theorem 2.8 in \cite{W}, $(i)$ and $(ii)$ follow by $(a)$ and $(c)$ in Theorem 2.8 in \cite{W}. In view of \eqref{cv0}, \eqref{gammav0}, for $\e=\e_{n}:=\frac{1}{n^2}$, we can find $\gamma_{n}\in \Gamma_{V_{0}}$ such that $\sup_{t\in [0, 1]}L_{V_{0}}(\gamma_{n}(t))\leq c_{V_{0}}+\frac{1}{n^{2}}$. Set $\tilde{\gamma}_{n}(t):=(0, \gamma_{n}(t))$. Then
$$
\sup_{t\in [0, 1]} (L_{V_{0}}\circ \Phi)(\tilde{\gamma}_{n}(t))=\sup_{t\in [0, 1]} L_{V_{0}}(\gamma_{n}(t))\leq c_{V_{0}}+\frac{1}{n^{2}}.
$$
From $(b)$ of Theorem 2.8 in \cite{W}, there exists $(\theta_{n}, u_{n})\in \R\times \X$ such that 
$$
{\rm dist}_{\R\times \X}((\theta_{n}, u_{n}), (0, \gamma_{n}(t)))\leq \frac{2}{n},
$$
that is $(iii)$ holds true. Here, we used the notation
$$
{\rm dist}_{\R\times \X}((\theta, u), A):=\inf_{(\tau, v)\in \R\times \X} (|\theta-\tau|^{2}+\|u-v\|^{2}_{\X})^{\frac{1}{2}},
$$
for $A\subset \R\times \h$.
Now, for $(h,w)\in  \R\times \X$, it holds
\begin{align}\label{2.13HLP}
\langle (L_{V_{0}}\circ \Phi)'(\theta_{n}, u_{n}), (h,w)\rangle =\langle L'_{V_{0}}(\Phi(\theta_{n}, u_{n})), \Phi'(\theta_{n}, w)\rangle+P(\Phi(\theta_{n}, u_{n}))h.
\end{align}
Then, choosing $h=1$ and $w=0$ in \eqref{2.13HLP}, we deduce that
$$
P(\Phi(\theta_{n}, u_{n}))\ri 0.
$$
On the other hand, for every $v\in \X$, taking $w(x,y)=v(e^{\theta_{n}}x, e^{\theta_{n}}y )$ and $h=0$ in \eqref{2.13HLP}, it follows from $(ii)$ and $(iii)$ that 
$$
\langle L_{V_{0}}'(\Phi(\theta_{n}, u_{n})), v\rangle=o(1)\|v(e^{\theta_{n}}x, e^{\theta_{n}}y )\|_{\X}=o(1)\|v\|_{\X}.
$$
Consequently, $w_{n}:=\Phi(\theta_{n}, u_{n})$ is the sequence that fulfills the desired properties.
\end{proof}

\begin{lem}\label{lem3.6HLP}
Every sequence $(w_{n})$ satisfying \eqref{3.6HLP} is bounded in $\X$.
\end{lem}
\begin{proof}
Using \eqref{3.6HLP} we see that
\begin{align*}
c_{V_{0}}+o_{n}(1)&=L_{V_{0}}(w_{n})-\frac{1}{N}P(w_{n})\\
&=\frac{1}{2} \wh{M}\left(\|w_{n}\|^{2}_{\x}\right) - \left(\frac{N-2s}{2N} \right) M\left( \|w_{n}\|^{2}_{\x} \right) \|w_{n}\|^{2}_{\x}.
\end{align*}
From $(M2)$ we deduce that $(\|w_{n}\|_{\x})$ is bounded in $\R$.
On the other hand, $P(w_{n})=o_{n}(1)$ and $(f_1)$-$(f_2)$ yield 
\begin{align*}
\frac{N-2s}{2}M\left( \|w_{n}\|^{2}_{\x} \right) \|w_{n}\|^{2}_{\x}+N\frac{V_{0}}{2}|w_{n}(\cdot, 0)|_{2}^{2}&=N\int_{\R^{N}}F(w_{n}(x,0))\, dx+o_{n}(1)\\
&\leq N\delta |w_{n}(\cdot, 0)|_{2}^{2}+NC_{\delta}|w_{n}(\cdot, 0)|_{\2}^{\2}+o_{n}(1).
\end{align*}
Choosing $\delta>0$ sufficiently small and using $(M1)$ and the boundedness of $(|w_{n}(\cdot, 0)|_{\2})$, we can infer that $(|w_{n}(\cdot, 0)|_{2})$ is bounded in $\R$. In conclusion, $(w_{n})$ is bounded in $\X$.
\end{proof}

\begin{lem}\label{lemmaB}
There exist a sequence $(x_{n})\subset \R^{N}$ and constants $R>0$, $\beta>0$ such that
$$
\int_{\Gamma_{R}^{0}(x_{n})} w^{2}_{n}(x, 0)\, dx\geq \beta,
$$
where $(w_{n})$ is the sequence given in Proposition \ref{prop3.4HLP}.
\end{lem}
\begin{proof}
Assume by contradiction that the thesis is not true. Then, by the vanishing Lions-type lemma (see Lemma 3.3 in \cite{HZ}), we deduce that 
\begin{align}\label{lions}
w_{n}(\cdot, 0)\ri 0 \mbox{ in } L^{q}(\R^{N}) \quad  \forall q\in (2, \2).
\end{align}
Consequently, by $(f_1)$-$(f_2)$, we have
$$
\int_{\R^{N}} f(w_{n}(x,0))w_{n}(x,0)\, dx=o_{n}(1).
$$
Recalling that $\langle L'_{V_{0}}(w_{n}), w_{n}\rangle=o_{n}(1)$, we get
$$
M(\|w_{n}\|^{2}_{\x}) \|w_{n}\|^{2}_{\x}+V_{0}|w_{n}(\cdot, 0)|_{2}^{2}=o_{n}(1)
$$
and using $(M1)$ we obtain that 
$$
\|w_{n}\|_{\X}\ri 0.
$$
Therefore, $L_{V_{0}}(w_{n})\ri 0$ and this leads to a contradiction because $c_{V_{0}}>0$.
\end{proof}

\noindent
Now we define
$$
\mathcal{T}_{V_{0}}:=\{u\in \X\setminus\{0\}: L'_{V_{0}}(u)=0, \max_{\R^{N}} u(\cdot,0)=u(0,0) \},
$$
$$
b_{V_{0}}:=\inf_{u\in \mathcal{T}_{V_{0}}} L_{V_{0}}(u), 
$$
and
$$
\mathcal{S}_{V_{0}}:=\{u\in \mathcal{T}_{V_{0}}: L_{V_{0}}(u)=b_{V_{0}}\}.
$$
\begin{lem}\label{Final}
Assume $(M1)$-$(M5)$. Then there exists $u\in \mathcal{S}_{V_{0}}$.
\end{lem}
\begin{proof}
Let $(w_{n})$ be the sequence given by Lemma \ref{prop3.4HLP}. Set $\tilde{w}_{n}(x,y):=w_{n}(x+x_{n},y)$ where $(x_{n})$ is given in Lemma \ref{lemmaB}.  By Lemma \ref{lem3.6HLP}, we know that $(w_{n})$ is bounded in $\Xr$, that is $\|w_{n}\|_{\X}\leq C$ for all $n\in \mathbb{N}$. Hence $\tilde{w}_{n}\rightharpoonup \tilde{w}$ in $X^{1,s}_{rad}(\R^{N+1}_{+})$ and $\tilde{w}_{n}(\cdot, 0)\ri \tilde{w}(\cdot, 0)$ in $L^{q}(\R^{N})$ for any $q\in (2, 2^{*}_{s})$, for some $\tilde{w}\in \Xr\setminus\{0\}$.
Then, $\tilde{w}$ is a weak solution to
\begin{align}\label{AMP}
\left\{
\begin{array}{ll}
-\dive(y^{1-2s} \nabla \tilde{w})=0 &\mbox{ in } \R^{N+1}_{+}, \\
\frac{1}{\alpha_{0}} \frac{\partial \tilde{w}}{\partial \nu^{1-2s}}=-V_{0} \tilde{w}(\cdot, 0)+f(\tilde{w}(\cdot, 0)) &\mbox{ in } \R^{N}, 
\end{array}
\right.
\end{align}
where 
$$
\alpha_{0}:=\lim_{n\ri \infty} M(\|\tilde{w}_{n}\|^{2}_{\x})=\lim_{n\ri \infty} M(\|w_{n}\|^{2}_{\x})\leq M(C^{2})<\infty.
$$
Note that the last inequality is due to $(M4)$.\\
Clearly, by Fatou's Lemma, we have
\begin{align}\label{m0A}
0<m_{0}\leq M(\|\tilde{w}\|^{2}_{\x})\leq \alpha_{0}.
\end{align}
In what follows, we prove that 
$$
\alpha_{0}=M(\|\tilde{w}\|^{2}_{\x}),
$$ 
and thus $\tilde{w}$ is a weak solution to \eqref{P}.
Since $\tilde{w}$ solves \eqref{AMP} and using the regularity assumptions on $f$, we deduce that $\tilde{w}$ satisfies the following Pohozaev identity \cite{Aade, BKS, CW}:
\begin{align}\label{POHw}
\frac{N-2s}{2} \alpha_{0} \|\tilde{w}\|^{2}_{\x}-N\int_{\R^{N}} \left(F(\tilde{w}(x,0))-\frac{V_{0}}{2} \tilde{w}^{2}(x,0)\right)\, dx=0.
\end{align}
Now, we apply Lemma 2.4 in \cite{CW} with $X=H^{s}_{\rm rad}(\R^{N})$, $P(t)=f(t)t$, $p_{1}=2$ and $p_{2}=2^{*}_{s}$ to see that
\begin{align*}
\alpha_{0}\|\tilde{w}\|^{2}_{\x}+V_{0}|\tilde{w}(\cdot, 0)|^{2}_{2}&\leq\liminf_{n\ri \infty} [M(\|\tilde{w}_{n}\|^{2}_{\x})\|\tilde{w}_{n}\|_{\x}^{s}+V_{0} |\tilde{w}_{n}(\cdot, 0)|_{2}^{2}] \\
&\leq \limsup_{n\ri \infty} [M(\|\tilde{w}_{n}\|^{2}_{\x})\|\tilde{w}_{n}\|_{\x}^{s}+V_{0} |\tilde{w}_{n}(\cdot, 0)|_{2}^{2}]  \\
&=\limsup_{n\ri \infty} [M(\|w_{n}\|^{2}_{\x})\|w_{n}\|^{2}_{\x}+V_{0} |w_{n}(\cdot, 0)|_{2}^{2}] \\
&=\limsup_{n\ri \infty} \int_{\R^{N}} f(w_{n}(x,0))w_{n}(x,0)\, dx\\
&=\lim_{n\ri \infty} \int_{\R^{N}} f(\tilde{w}_{n}(x,0))\tilde{w}_{n}(x,0)\, dx\\
&=\int_{\R^{N}} f(\tilde{w}(x,0)) \tilde{w}(x,0)\, dx \\
&=\alpha_{0}\|\tilde{w}\|^{2}_{\x}+V_{0}|\tilde{w}(\cdot, 0)|^{2}_{2}
\end{align*}
which implies that $\|\tilde{w}_{n}\|_{\X}\ri \|\tilde{w}\|_{\X}$ and thus $\tilde{w}_{n}\ri \tilde{w}$ in $\X$. Hence, $\alpha_{0}=M(\|\tilde{w}\|^{2}_{\x})$.
Therefore, by $L_{V_{0}}(w_{n})=L_{V_{0}}(\tilde{w}_{n})\ri c_{V_{0}}$ and $L'_{V_{0}}(w_{n})=L'_{V_{0}}(\tilde{w}_{n})\ri 0$, we have that $L_{V_{0}}(\tilde{w})=c_{V_{0}}$ and $L'_{V_{0}}(\tilde{w})=0$. Since $\tilde{w}\neq 0$, we deduce that $c_{V_{0}}\geq b_{V_{0}}$. 

Now, let $w\in \X\setminus\{0\}$ be any solution to \eqref{LP}.
Define
\begin{equation*}
\gamma(t):=\left\{
\begin{array}{ll}
w(\frac{x}{t}, \frac{y}{t}) &\mbox{ for } t>0, \\
0  &\mbox{ for } t=0.
\end{array}
\right.
\end{equation*}
Using the fact that $w$ satisfies the Pohozaev identity (see Lemma \ref{lem2.11FIJ}), we get
\begin{align*}
L_{V_{0}}(\gamma(t))&=\frac{1}{2} \wh{M}\left(t^{N-2s} \|w\|^{2}_{\x}\right) -t^{N} \left(\frac{N-2s}{2N} \right) M\left(\|w\|^{2}_{\x}\right) \|w\|^{2}_{\x},
\end{align*}
and differentiating with respect to $t$ we obtain
\begin{align*}
\frac{d}{dt} L_{V_{0}}(\gamma(t))&=\frac{N-2s}{2}\|w\|^{2}_{\x} t^{N-2s-1}\left[M(t^{N-2s}\|w\|^{2}_{\x})-t^{2s} M(\|w\|^{2}_{\x})\right]. 
\end{align*}
By $(M5)$ and using a change of variable, we observe that $t\mapsto M(t^{N-2s} \|w\|^{2}_{\x})/t^{2s}$ is nonincreasing in $(0, \infty)$, so we have
$$
\frac{d}{dt} L_{V_{0}}(\gamma(t))> 0 \quad \forall t\in (0, 1), \, \frac{d}{dt} L_{V_{0}}(\gamma(t))< 0 \quad \forall t\in (1, \infty),
$$
which implies that 
$$
\max_{t\geq 0} L_{V_{0}}(\gamma(t))=L_{V_{0}}(\gamma(1))=L_{V_{0}}(w).
$$
Moreover, noting that $(M1)$ and $(M3)$ yield
$$
\lim_{t\ri \infty} \frac{\wh{M}(t^{N-2s})}{t^{N}}=\left[\frac{\infty}{\infty}\right]=\lim_{t\ri \infty} \frac{M(t^{N-2s})}{(t^{N-2s})^{\frac{2s}{N-2s}}} \frac{N-2s}{N}=0,
$$
we deduce
\begin{align*}
L_{V_{0}}(\gamma(t))=\frac{t^{N}}{2}\left[\frac{1}{t^{N}}\wh{M}\left(t^{N-2s} \|w\|^{2}_{\x}\right)- \left(\frac{N-2s}{2N} \right) M\left(\|w\|^{2}_{\x}\right) \|w\|^{2}_{\x} \right]\ri -\infty,
\end{align*}
as $t\ri \infty$. Then there exists $\tau>0$ sufficiently large such that $L_{V_{0}}(\gamma(\tau))<0$. After a suitable scale change in $t$, we obtain that $\gamma\in \Gamma_{V_{0}}$.
By the definition of $c_{V_{0}}$, we see that $L_{V_{0}}(w)\geq c_{V_{0}}$. Since $w$ is arbitrary, we have that $b_{V_{0}}\geq c_{V_{0}}$ and this implies that $b_{V_{0}}=c_{V_{0}}$.

Choosing $u^{-}=\min\{u, 0\}$ as test function in the weak formulation of \eqref{LP} we can deduce that $u\geq 0$ in $\R^{N}$. By $(f_1)$-$(f_2)$ and using a Moser iteration argument (see \cite{Aade, CW}), we obtain that $u\in L^{\infty}(\R^{N})$. By the growth assumptions on $f$ and in view of the H\"older regularity results in \cite{Silvestre}, we deduce that $u\in C^{0, \beta}(\R^{N})$ (see \cite{Aade, BKS, CW}). From the Harnack inequality \cite{CSire, JLX} we conclude that $u>0$ in $\R^{N}$.
\end{proof}

\begin{remark}\label{REMARK}
For $m>0$, we use the notation 
$$
L_{m}(u)=\frac{1}{2}\wh{M}(\|u\|^{2}_{\x})+\frac{m}{2}|u(\cdot, 0)|^{2}_{2}-\int_{\R^{N}} F(u(x, 0))\, dx
$$
and denote by $c_{m}$ the corresponding mountain pass level. It is standard to verify that if $m_{1}>m_{2}$ then $c_{m_1}>c_{m_2}$.
\end{remark}

\noindent
In what follows, we aim to show that $\mathcal{S}_{V_{0}}$ is compact in $\X$. To do this we begin by giving some auxiliary results.
Let us consider the following fractional elliptic problem:
\begin{align}\label{PSchrodinger}
\left\{
\begin{array}{ll}
-\dive(y^{1-2s} \nabla w)=0 &\mbox{ in } \R^{N+1}_{+}, \\
\frac{\partial w}{\partial \nu^{1-2s}}=-V_{0} w(\cdot, 0)+f(w(\cdot, 0)) &\mbox{ in } \R^{N}. 
\end{array}
\right.
\end{align}
If $w$ is a solution to \eqref{PSchrodinger}, then it satisfies the Pohozaev identity (see \cite{Adie, Aade, BKS, CW, ZDOS})
\begin{equation}\label{2.20FIJ}
\frac{N-2s}{2}\|w\|^{2}_{\x}-N\int_{\R^{N}} F(u(x,0))-\frac{V_{0}}{2} u^{2}(x,0)\, dx=0.
\end{equation}
Let 
$$
\E_{V_{0}}(u)=\frac{1}{2} \|u\|^{2}_{\x}+\frac{V_{0}}{2}|u(\cdot, 0)|^{2}_{2}-\int_{\R^{N}} F(u(x,0))\, dx,
$$
$$
\tilde{b}_{V_{0}}:=\inf_{u\in \widetilde{\mathcal{T}}_{V_{0}}} \E_{V_{0}}(u), 
$$
$$
\widetilde{\mathcal{T}}_{V_{0}}=\{u\in \X\setminus\{0\}: \E'_{V_{0}}(u)=0, \max_{\R^{N}} u(\cdot,0)=u(0,0) \},
$$
and
$$
\widetilde{\mathcal{S}}_{V_{0}}=\{u\in \widetilde{\mathcal{T}}_{V_{0}}: \E_{V_{0}}(u)=\tilde{b}_{V_{0}}\}.
$$
Next we show that it is possible to define a map which relates the ground state solutions of \eqref{PSchrodinger} to the ones for \eqref{LP}.
We first prove the following result for the Kirchhoff functions.
\begin{lem}\label{lem2.17FIJ}
Assume that $M\in C([0, \infty))$ and $M(t)\geq 0$. Then, $(M5)$ is equivalent to
\begin{compactenum}[$(M6)$]
\item $t\mapsto \wh{M}(t)-\left(1-\frac{2s}{N}\right)M(t)t$  is nondecreasing in  $[0, \infty)$.
\end{compactenum}
\end{lem}
\begin{proof}
We argue as in Lemma 2.17 in \cite{FIJ}. Let $(M5)$ be in force. Then, for $0\leq t_1<t_{2}$ we have
\begin{align}\begin{split}\label{2.21FIJ}
\wh{M}(t_{2})-\left(1-\frac{2s}{N} \right)M(t_{2})t_{2}&=\wh{M}(t_{1})+\int_{t_1}^{t_2} \frac{M(t)}{t^{\frac{2s}{N-2s}}}t^{\frac{2s}{N-2s}}\, dt-\left(1-\frac{2s}{N} \right)M(t_{2})t_{2}\\
&\geq \wh{M}(t_{1})+ \frac{M(t_2)}{t_{2}^{\frac{2s}{N-2s}}} \int_{t_1}^{t_2} t^{\frac{2s}{N-2s}}\, dt-\left(1-\frac{2s}{N} \right)M(t_{2})t_{2}\\
&=\wh{M}(t_{1})-\left(1-\frac{2s}{N} \right) \frac{M(t_2)}{t_{2}^{\frac{2s}{N-2s}}} t_{1}^{\frac{N}{N-2s}}  \\
&\geq \wh{M}(t_{1})-\left(1-\frac{2s}{N} \right)M(t_{1})t_{1}. 
\end{split}\end{align}
The other implication is obtained as in the case $s=1$ with small modifications, so we omit the details.
\end{proof}

\begin{lem}\label{lem2.16FIJ}
Assume $(M1)$-$(M5)$. Then, $\mathcal{S}_{V_{0}}\neq \emptyset$ and there exists an injective map $T: \widetilde{\mathcal{S}}_{V_{0}} \ri \mathcal{S}_{V_{0}}$.
\end{lem}
\begin{proof}
By \cite{Aade, BKS, CW} we know that $\widetilde{\mathcal{S}}_{V_{0}}\neq \emptyset$. Let $\phi\in \widetilde{\mathcal{S}}_{V_{0}}$ and define
$$
t_{\phi}:=\inf\left\{ t>0: t^{2s}=M(t^{N-2s} \|\phi\|^{2}_{\x})\right\}.
$$
In what follows we verify that $t_{\phi}\in (0, \infty)$. Since $\T_{V_{0}}\neq \emptyset$ by Lemma \ref{Final}, we can find $w\in \T_{V_{0}}$  and put $\alpha^{2s}:=M(\|w\|^{2}_{\x})$. Set $w_{\alpha}(x,y)=w(\alpha x, \alpha y)$ and note that $w_{\alpha}$ is a weak solution to 
\begin{align}\label{Palpha}
\left\{
\begin{array}{ll}
-\dive(y^{1-2s} \nabla w_{\alpha})=0 &\mbox{ in } \R^{N+1}_{+}, \\
\frac{\partial w_{\alpha}}{\partial \nu^{1-2s}}=-V_{0} w_{\alpha}(\cdot, 0)+f(w_{\alpha}(\cdot, 0)) &\mbox{ in } \R^{N}. 
\end{array}
\right.
\end{align}
By \eqref{2.20FIJ} we get
$$
\frac{s}{N} \|\phi\|^{2}_{\x}=\E_{V_{0}}(\phi)_{V_{0}}\leq \E_{V_{0}}(w_{\alpha})=\frac{s}{N} \|w_{\alpha}\|^{2}_{\x}=\frac{s}{N} \alpha^{2s-N} \|w\|^{2}_{\x}
$$
that is $\alpha^{N-2s} \|\phi\|^{2}_{\x}\leq \|w\|^{2}_{\x}$. Using $(M4)$ we have 
$$
M(\alpha^{N-2s} \|\phi\|^{2}_{\x})\leq M( \|w\|^{2}_{\x})=\alpha^{2s}.
$$ 
From $(M1)$ and the continuity of $M$, there is $t_{0}\in (0, \alpha]$ such that $t_{0}^{2s}=M(t_{0}^{N-2s}\|\phi\|^{2}_{\x})$. Consequently, $0<m_{0}\leq t_{\phi}^{2s}\leq \alpha^{2s}$ and $t_{\phi}$ is well-defined.

At this point, for $u\in \T_{V_{0}}$, we define 
$$
(Tu)(x,y):=u(x/t_{u}, y/t_{u}).
$$ 
Since 
$$
t_{u}^{2s}=M(t_{u}^{N-2s}\|u\|^{2}_{\x}),
$$ 
we see that $Tu$ is a solution to \eqref{LP}. Using $t_{u}\leq \alpha$ and $\alpha^{N-2s} \|u\|^{2}_{\x}\leq \|w\|^{2}_{\x}$ we get $\|Tu\|^{2}_{\x}\leq \|w\|^{2}_{\x}$. 
On the other hand, we observe that for all $u\in \X$ such that $P(u)=0$ it holds
$$
L_{V_{0}}(u)=\frac{1}{2} \left[\wh{M}(\|u\|^{2}_{\x})-\left( 1-\frac{2s}{N} \right) M(\|u\|^{2}_{\x})\|u\|^{2}_{\x}  \right].
$$
Then, from Lemma \ref{lem2.17FIJ} and $(M5)$, we deduce that $L_{V_{0}}(Tu)\leq L_{V_{0}}(w)$. By the arbitrariness of $w\in \T_{V_{0}}$, we infer that $Tu\in \mathcal{S}_{V_{0}}$. Hence, $\mathcal{S}_{V_{0}}\neq \emptyset$ and $T: \widetilde{\mathcal{S}}_{V_{0}} \ri \mathcal{S}_{V_{0}}$ is well-defined. \\
Finally, we show that  $T$ is injective. Let $u_{1}, u_{2}\in \widetilde{\mathcal{S}}_{V_{0}}$ be such that $Tu_{1}=Tu_{2}$. Then, $u_{1}(x, y)=u_{2}(\alpha x, \alpha y)$ for some $\alpha>0$. Since $u_{1}(\cdot , 0)$ and $u_{2}(\cdot, 0)$ are nontrivial solutions of 
$(-\Delta)^{s}u+V_{0}u=f(u)$ in $\R^{N}$, we deduce that 
$\alpha^{2s}(-\Delta)^{s}u_{2}(\alpha x, 0)=(-\Delta)^{s} u_{1}(x, 0)=(-\Delta)^{s}u_{2}(\alpha x, 0)$ which implies that $(\alpha^{2s}-1)(-\Delta)^{s}u_{2}(\cdot, 0)=0$ in $\R^{N}$. Hence, $\alpha=1$ and $u_{1}\equiv u_{2}$. 
\end{proof}

\begin{prop}\label{COMPs}
$\mathcal{S}_{V_{0}}$ is compact in $\X$.
\end{prop}
\begin{proof}
Let $(w_{n})\subset \mathcal{S}_{V_{0}}$ and set $v_{n}(x, y):=w_{n}(\alpha_{n} x, \alpha_{n} y)$ where 
$$
\alpha_{n}^{2s}:=M(\|w_{n}\|^{2}_{\x}).
$$ 
Then, $v_{n}$ is a solution to \eqref{PSchrodinger}. Now we prove that $v_{n}\in \widetilde{\mathcal{S}}_{V_{0}}$ and that there exists $C_{0}>0$ such that $m_{0}\leq \alpha_{n}^{2s}\leq C_{0}^{2s}$ for all $n\in \mathbb{N}$. Note that $m_{0}\leq \alpha_{n}^{2s}$ thanks to $(M1)$. Now, by Lemma \ref{lem2.11FIJ} we have 
\begin{align*}
b_{V_{0}} &= L_{V_{0}} (w_{n}) - \frac{1}{N} P(w_{n}) \\
&=\frac{1}{2} \left[ \wh{M}\left(\|w_{n}\|^{2}_{\x}\right) - \left(1- \frac{2s}{N}\right) M\left(\|w_{n}\|^{2}_{\x}\right) \|w_{n}\|^{2}_{x}\right].  
\end{align*} 
In light of $(M2)$ we deduce that $\|w_{n}\|_{\x}$ is bounded and then $(\alpha_{n})$ is bounded. \\
Take $\phi_{n}\in \widetilde{\mathcal{S}}_{V_{0}}$. Proceeding as in the proof of Lemma \ref{lem2.16FIJ} and using $(M6)$ we can see that $\|\phi_{n}\|^{2}_{\x} \leq \|v_{n}\|^{2}_{\x}$, $t_{n}\leq \alpha_{n}$ and $b_{V_{0}}= L_{V_{0}}(\phi_{n, t_{n}})\leq L_{V_{0}}(w_{n})= b_{V_{0}}$, where 
$$
t_{n}:=\inf \left\{t\in (0, \infty) : t^{2s} = M(t^{N-2s} \|\phi_{n}\|^{2}_{\x})\right\}
$$ 
and $\phi_{n, t_{n}}(x, y):= \phi_{n}(\frac{x}{t_{n}}, \frac{y}{t_{n}})= T(\phi_{n})$. Moreover, $L_{V_{0}}(\phi_{n, t_{n}})= b_{V_{0}}= L_{V_{0}}(w_{n})$. At this point, if we show that 
\begin{align}\label{pizza}
\|\phi_{n}\|_{\x}= \|v_{n}\|_{\x},
\end{align} 
then we have 
\begin{align*}
\E_{V_{0}}(\phi_{n})= \frac{s}{N}\|\phi_{n}\|^{2}_{\x}= \frac{s}{N} \|v_{n}\|^{2}_{\x}= \E_{V_{0}}(v_{n}), 
\end{align*} 
where we used \eqref{2.20FIJ}. Hence we deduce that $v_{n}\in \widetilde{\mathcal{S}}_{V_{0}}$. Next, we prove that \eqref{pizza} holds true. Assume by contradiction that $\|v_{n}\|_{\x}>\|\phi_{n}\|_{\x}$. Taking into account that $t_{n}\leq \alpha_{n}$ and $\|w_{n}\|^{2}_{\x}= \alpha_{n}^{N-2s} \|v_{n}\|^{2}_{\x}$, we get
\begin{align*}
\|\phi_{n, t_{n}}\|^{2}_{\x}= t_{n}^{N-2s} \|\phi_{n}\|^{2}_{\x} <\alpha_{n}^{N-2s} \|v_{n}\|^{2}_{\x}= \|w_{n}\|^{2}_{\x}. 
\end{align*}
On the other hand, using $P(\phi_{n, t_{n}})=0=P(w_{n})$, we infer that 
\begin{align*}
&\frac{1}{2} \left\{ \wh{M}(\|\phi_{n, t_{n}}\|^{2}_{\x} ) - \left(1- \frac{2s}{N}\right) M(\|\phi_{n, t_{n}}\|^{2}_{\x}) \|\phi_{n, t_{n}}\|^{2}_{\x}\right\}\\
&=L_{V_{0}}(\phi_{n, t_{n}}) = L_{V_{0}}(w_{n}) = \frac{1}{2} \left\{ \wh{M}(\|w_{n}\|^{2}_{\x} ) - \left(1- \frac{2s}{N}\right) M(\|w_{n}\|^{2}_{\x}) \|w_{n}\|^{2}_{\x}\right\}. 
\end{align*}
By $(M5)$, $(M6)$ in Lemma \ref{lem2.16FIJ} and \eqref{2.21FIJ}, it is easy to see that for any $\|\phi_{n, t_{n}}\|^{2}_{\x}\leq t_{1}<t_{2}\leq \|w_{n}\|^{2}_{\x}$ it holds 
\begin{align*}
\wh{M}(t_{1}) - \left(1- \frac{2s}{N}\right) M(t_{1}) t_{1}= \wh{M}(t_{2}) - \left(1- \frac{2s}{N}\right) M(t_{2}) t_{2}
\end{align*} 
and 
\begin{align}\label{2.23FIJ}
\frac{M(t_{1})}{t_{1}^{2s/(N-2s)}} = \frac{M(t_{2})}{t_{2}^{2s/(N-2s)}}. 
\end{align}
Otherwise, we have $L_{V_{0}}(\phi_{n, t_{n}})< L_{V_{0}}(w_{n})$, that is a contradiction. Moreover, in view of \eqref{2.23FIJ}, we get
\begin{align*}
M(t)= k_{0}t^{\frac{2s}{N-2s}} \mbox{ in } [\|\phi_{n, t_{n}}\|^{2}_{\x}, \|w_{n}\|^{2}_{\x}], 
\end{align*} 
for some $k_{0}>0$. By the definitions of $\alpha_{n}$ and $t_{n}$, and using $t_{n}^{N-2s} \|\phi_{n}\|^{2}_{\x}= \|\phi_{n, t_{n}}\|^{2}_{\x}$, we deduce that
\begin{align*}
t_{n}^{2s}&=M(t_{n}^{N-2s}\|\phi_{n}\|^{2}_{\x})=k_{0} t_{n}^{2s} \|\phi_{n}\|^{\frac{4s}{N-2s}}_{\x}\\
\alpha_{n}^{2s}&=M(\|w_{n}\|^{2}_{\x})=M(\|v_{n}\|^{2}_{\x})=k_{0} \alpha_{n}^{2s} \|v_{n}\|^{\frac{4s}{N-2s}}_{\x} 
\end{align*}
which gives $ \|\phi_{n}\|^{2}_{\x}=k_{0}^{-\frac{N-2s}{2}}=\|v_{n}\|^{2}_{\x}$ and this is a contradiction.

Now, observing that $w_{n}(x,y)=v_{n}(\alpha_{n}^{-1} x, \alpha_{n}^{-1} y)$, it is enough to prove that $v_{n}$ has a convergent subsequence in $\X$. Since $\mathcal{S}_{V_{0}}$ is compact in $\X$ (see Proposition 2.6 in \cite{Seok}) we obtain the thesis.
\end{proof}

\section{critical limiting problems}

In this section we extend the previous results for the following critical limiting problem:
\begin{align}\label{LPC}
\left\{
\begin{array}{ll}
-\dive(y^{1-2s} \nabla w)=0 &\mbox{ in } \R^{N+1}_{+}, \\
\frac{1}{M(\|w\|^{2}_{\x})} \frac{\partial w}{\partial \nu^{1-2s}}=-V_{0} w(\cdot, 0)+f(w(\cdot, 0)) &\mbox{ in } \R^{N},
\end{array}
\right.
\end{align}
where $f$ satisfies $(f_1)$, $(f'_2)$ and $(f'_3)$.
The study of \eqref{LPC} will be done following some arguments used in \cite{ZCDO}.\\
In order to find weak solutions to \eqref{LPC}, we look for critical points of the energy functional $L_{V_{0}}: \X\ri \R$ given by
$$
L_{V_{0}}(u):=\frac{1}{2} \wh{M}\left(\|u\|^{2}_{X^{s}(\R^{N+1}_{+})}\right)+\frac{1}{2}\int_{\R^{N}} V_{0} u^{2}(x,0)\, dx-\int_{\R^{N}} F(u(x,0))\, dx.
$$
We define
$$
\mathcal{T}_{V_{0}}:=\left\{u\in \X\setminus\{0\}: L'_{V_{0}}(u)=0, \max_{\R^{N}} u(\cdot,0)=u(0,0) \right\},
$$
$$
b_{V_{0}}:=\inf_{u\in \mathcal{T}_{V_{0}}} L_{V_{0}}(u), 
$$
and
$$
\mathcal{S}_{V_{0}}:=\{u\in \mathcal{T}_{V_{0}}: L_{V_{0}}(u)=b_{V_{0}}\}.
$$
We consider the following elliptic critical problem:
\begin{align}\label{PCSchrodinger}
\left\{
\begin{array}{ll}
-\dive(y^{1-2s} \nabla w)=0 &\mbox{ in } \R^{N+1}_{+}, \\
\frac{\partial w}{\partial \nu^{1-2s}}=-V_{0} w(\cdot, 0)+f(w(\cdot, 0)) &\mbox{ in } \R^{N}. 
\end{array}
\right.
\end{align}
Any solution $w$ to \eqref{PCSchrodinger} satisfies the following Pohozaev identity (see \cite{Adie, JLZ, ZDOS})
\begin{equation}\label{2.20FIJ}
\frac{N-2s}{2}\|w\|^{2}_{\x}-N\int_{\R^{N}} F(u(x,0))-\frac{V_{0}}{2} u^{2}(x,0)\, dx=0.
\end{equation}
Let us define
$$
\E_{V_{0}}(u):=\frac{1}{2} \|u\|^{2}_{\x}+\frac{V_{0}}{2}|u(\cdot, 0)|^{2}_{2}-\int_{\R^{N}} F(u(x,0))\, dx,
$$
$$
\tilde{b}_{V_{0}}:=\inf_{u\in \widetilde{\mathcal{T}}_{V_{0}}} \E_{V_{0}}(u), 
$$
where
$$
\widetilde{\mathcal{T}}_{V_{0}}:=\left\{u\in \X\setminus\{0\}: \E'_{V_{0}}(u)=0, \max_{\R^{N}} u(\cdot,0)=u(0,0) \right\},
$$
and
$$
\widetilde{\mathcal{S}}_{V_{0}}:=\{u\in \widetilde{\mathcal{T}}_{V_{0}}: \E_{V_{0}}(u)=\tilde{b}_{V_{0}}\}.
$$

In what follows, we show that $\mathcal{S}_{V_{0}}$ is compact in $\X$.
Arguing as in the proof of Lemma \ref{lem2.16FIJ} and in view of results in \cite{Adie, ZDOS}, we obtain that:
\begin{lem}\label{lem2.1ZCDO}
Assume $(M1)$-$(M5)$. Then, $\mathcal{S}_{V_{0}}\neq \emptyset$ if $\widetilde{\mathcal{S}}_{V_{0}}\neq \emptyset$. Moreover, there exists an injective map $T: \widetilde{\mathcal{S}}_{V_{0}} \ri \mathcal{S}_{V_{0}}$. In particular, for any $u\in \widetilde{\mathcal{S}}_{V_{0}}$, 
$$
(Tu)(x, y):=u(x/t_{u}, y/t_{u})
$$ 
where $t_{u}:=\inf\left\{t\in (0, \infty): t^{2s}=M(t^{N-2s}\|u\|^{2}_{\x})\right\}$.
\end{lem}

\begin{lem}\label{lem2.2ZCDO}
Assume that $\widetilde{\mathcal{S}}_{V_{0}}\neq \emptyset$. Then $\mathcal{S}_{V_{0}}\neq \emptyset$. Moreover, for any $v\in \mathcal{S}_{V_{0}}$ there exists $u\in \widetilde{\mathcal{S}}_{V_{0}}$ such that $v(x, y)=u(x/h_{v}, y/h_{v})$, where
$h_{v}^{2s}=M(\|v\|^{2}_{\x})$.
\end{lem}
\begin{proof}
By the definition of $T$, we know that $\mathcal{S}_{V_{0}}\neq \emptyset$ if $\widetilde{\mathcal{S}}_{V_{0}}\neq \emptyset$. Let $v\in \mathcal{S}_{V_{0}}$. Thus $v$ satisfies \eqref{LPC} and $L_{V_{0}}(v)=b_{V_{0}}$. 
Define $u(x, y):=v(h x, h y)$ where $h^{2s}:=M(\|v\|^{2}_{\x})$. Then, $u$ solves \eqref{PCSchrodinger}. Now, we show that $u\in \widetilde{\mathcal{S}}_{V_{0}}$. To do this, we prove that $\E_{V_{0}}(u)=\tilde{b}_{V_{0}}$. Using the Pohozaev identity, we know that
$$
\E_{V_{0}}(u)=\frac{s}{N} \left[\frac{M(\|v\|^{2}_{\x})}{(\|v\|^{2}_{\x})^{\frac{2s}{N-2s}}}\right]^{\frac{2s-N}{2s}}.
$$
Let $\tilde{u}\in \widetilde{\mathcal{S}}_{V_{0}}$. Then $\tilde{v}:=T\tilde{u}=u(x/t_{\tilde{u}}, y/t_{\tilde{u}})\in \mathcal{S}_{V_{0}}$, where $t_{\tilde{u}}$ is defined as in Lemma \ref{lem2.1ZCDO}. By Lemma \ref{lem2.11FIJ} (which holds even if replace $(f_2)$-$(f_3)$ by $(f'_2)$-$(f'_3)$), we obtain that
\begin{align}\begin{split}\label{2.3ZCDO}
L_{V_{0}}(\tilde{v})&=\frac{1}{2}\left[\wh{M}(\|\tilde{v}\|^{2}_{\x})-\left(1-\frac{2s}{N}\right)M(\|\tilde{v}\|^{2}_{\x})\|\tilde{v}\|^{2}_{\x}\right]=b_{V_{0}}  \\
L_{V_{0}}(v)&=\frac{1}{2}\left[\wh{M}(\|v\|^{2}_{\x})-\left(1-\frac{2s}{N}\right)M(\|v\|^{2}_{\x})\|v\|^{2}_{\x}\right]=b_{V_{0}}. 
\end{split}\end{align}
On the other hand, by the proof of Lemma \ref{lem2.17FIJ} and $(M5)$, it is easy to see that if for some $0\leq t_{1}<t_{2}$ it holds
\begin{align*}
\wh{M}(t_{1}) - \left(1- \frac{2s}{N}\right) M(t_{1}) t_{1}= \wh{M}(t_{2}) - \left(1- \frac{2s}{N}\right) M(t_{2}) t_{2}
\end{align*} 
then 
\begin{align*}
\frac{M(t_{1})}{t_{1}^{2s/(N-2s)}} = \frac{M(t_{2})}{t_{2}^{2s/(N-2s)}}. 
\end{align*}
Hence, by \eqref{2.3ZCDO}, it follows that
$$
\E_{V_{0}}(u)=\frac{s}{N}  \left[\frac{M(\|v\|^{2}_{\x})}{(\|v\|^{2}_{\x})^{\frac{2s}{N-2s}}}\right]^{\frac{2s-N}{2s}}=\frac{s}{N} \|\tilde{u}\|^{2}_{\x}=\tilde{b}_{V_{0}}
$$
that is $u\in \widetilde{\mathcal{S}}_{V_{0}}$.
\end{proof}

\begin{lem}\label{lem2.3ZCDO}
Assume that $\widetilde{\mathcal{S}}_{V_{0}}\neq \emptyset$. Then there exist $C, c>0$ (independent of $v$)  
such that $c\leq h_{v}\leq C$ for all $v\in \mathcal{S}_{V_{0}}$, where $h_{v}$ is given in Lemma \ref{lem2.2ZCDO}.
\end{lem}
\begin{proof}
Fix $v\in \mathcal{S}_{V_{0}}$. Then $h_{v}^{2s}=M(\|v\|^{2}_{\x})$. From $(M1)$ we have that $h_{v}^{2s}\geq m_{0}$. On the other hand, by Lemma \ref{lem2.11FIJ}, we see that for all $v\in \mathcal{S}_{V_{0}}$,
$$
L_{V_{0}}(v)=\frac{1}{2}\left[\wh{M}(\|v\|^{2}_{\x})-\left(1-\frac{2s}{N}\right)M(\|v\|^{2}_{\x})\|v\|^{2}_{\x}\right]=b_{V_{0}}.
$$ 
Thus, in view of $(M2)$, we infer that $\sup_{v\in \mathcal{S}_{V_{0}}} h_{v}<\infty$.
\end{proof}

\noindent
Now, we recall the following result (see \cite{Adie, He, JLZ}):
\begin{lem}\label{lem2.4ZCDO}
Assume that $(f_1)$, $(f'_2)$-$(f'_3)$ hold true. Then:
\begin{compactenum}[$(i)$]
\item there exists $u\in \widetilde{\mathcal{S}}_{V_{0}}$ such that $u(\cdot, 0)\in C^{1}(\R^{N})\cap L^{\infty}(\R^{N})$ and radially symmetric;
\item $\widetilde{\mathcal{S}}_{V_{0}}$ is compact in $\X$.
\end{compactenum}
\end{lem}

\noindent
As a consequence of Lemma \ref{lem2.2ZCDO}, Lemma \ref{lem2.3ZCDO} and Lemma \ref{lem2.4ZCDO}, we obtain that:
\begin{prop}\label{prop2.1ZCDO}
Under the assumptions of Theorem \ref{thm2} we have that:
\begin{compactenum}[$(i)$]
\item there exists $u\in \mathcal{S}_{V_{0}}$ such that $u(\cdot, 0)\in C^{1}(\R^{N})\cap L^{\infty}(\R^{N})$ and radially symmetric;
\item $\mathcal{S}_{V_{0}}$ is compact in $\X$.
\end{compactenum}
\end{prop}

\section{Proof of Theorem \ref{thm1}}
In light of Section $2$, to study \eqref{EP} we look for critical points of the functional $I_{\e}: X_{\e}\ri \R$ defined as
$$
I_{\e}(u)=\frac{1}{2}\wh{M}\left(\|u\|^{2}_{X^{s}(\R^{N+1}_{+})}\right)+\frac{1}{2}\int_{\R^{N}} V_{\e}(x) u^{2}(x, 0)\, dx-\int_{\R^{N}} F(u(x,0))\, dx
$$
where
$$
X_{\e}:=\left\{u\in X^{1,s}(\R^{N+1}_{+}): \int_{\R^{N}} V_{\e}(x) u^{2}(x, 0)\, dx<\infty\right\}
$$
endowed with the norm
$$
\|u\|_{\e}:=\left(\|u\|^{2}_{\x}+ \int_{\R^{N}} V_{\e}(x) u^{2}(x, 0)\, dx \right)^{\frac{1}{2}}.
$$
It follows from $(V_1)$ that $X_{\e}\subset X^{1,s}(\R^{N+1}_{+})$ and
$$
\|u\|_{X^{1,s}(\R^{N+1}_{+})}^{2}\leq \max\{1, V_{1}^{-1}\} \|u\|^{2}_{\e} \quad \forall u\in X_{\e}.
$$
We denote by $(X_{\e})^{-1}$ the dual space of $X_{\e}$ endowed with the norm $\|T\|_{(X_{\e})^{-1}}:=\sup\{Tu: u\in X_{\e}, \|u\|_{\e}\leq 1\}$.\\
In order to obtain some convergence results and consequently results of existence for small $\e>0$, we need to modify $f(t)$ once more. Namely, as in \cite{DF, Gloss}, we consider the following Carath\'eodory function
$$
g(x, t):=\chi_{\Lambda}(x) f(t)+(1-\chi_{\Lambda}(x)) \wh{f}(t) \quad \mbox{ for } (x,t)\in \R^{N}\times \R,
$$
and we write $G(x,t):=\int_{0}^{t} g(x, \tau)\, d\tau$, where $\chi_{\Lambda}$ denotes the characteristic function of $\Lambda$, and
\begin{equation*}
\wh{f}(t):=\left\{
\begin{array}{ll}
f(t) &\mbox{ for } t<a, \\
\min\{f(t), \frac{V_{1}}{2} t\}  &\mbox{ for } t\geq a,
\end{array}
\right.
\end{equation*}
where $a\in (0, \tau_{0})$ is such that $|f(t)|\leq \frac{V_{1}}{2} t$ for $t\in (0, a]$.
By $(f_1)$-$(f_2)$, it is easy to check that:
\begin{itemize}
\item $\lim_{t\ri 0} \frac{g(x,t)}{t}=\lim_{t\ri 0} \frac{f(t)}{t}=0$ uniformly in $x\in \R^{N}$,
\item $\limsup_{t\ri \infty} \frac{g(x,t)}{t^{p}}\leq \limsup_{t\ri \infty} \frac{f(t)}{t^p}<\infty$, for all $x\in \R^{N}$.
\end{itemize}

\noindent
Therefore, we consider the following modified problem:
\begin{align}\label{MEP}
\left\{
\begin{array}{ll}
-\dive(y^{1-2s} \nabla u)=0 &\mbox{ in } \R^{N+1}_{+}, \\
\frac{1}{M(\|u\|^{2}_{X^{s}(\R^{N+1}_{+})})} \frac{\partial u}{\partial \nu^{1-2s}}=-V_{\e} u(\cdot, 0)+g_{\e}(\cdot, u(\cdot, 0)) &\mbox{ in } \R^{N},
\end{array}
\right.
\end{align}
where we set $g_{\e}(x,t):=g(\e x, t)$.  Obviously, if $u_{\e}$ is a positive solution of \eqref{MEP} satisfying $u_{\e}(x,0)\leq a$ for $x\in \R^{N}\setminus \Lambda_{\e}$, then $u_{\e}$ is indeed a solution of \eqref{EP}.
Now, inspired by \cite{BJ, BW, FIJ, Gloss}, we define
$$
J_{\e}(u):=P_{\e}(u)+Q_{\e}(u)
$$
where
$$
P_{\e}(u):=\frac{1}{2}\wh{M}\left(\|u\|^{2}_{X^{s}(\R^{N+1}_{+})}\right)+\frac{1}{2}\int_{\R^{N}} V_{\e}(x) u^{2}(x,0)\, dx-\int_{\R^{N}} G_{\e}(x,u(x,0))\, dx
$$
and
$$
Q_{\e}(u):=\left( \int_{\R^{N}} \chi_{\e}(x) u^{2}(x,0)\, dx-1 \right)_{+}^{2}
$$
with
\begin{equation*}
\chi_{\e}(x):=\left\{
\begin{array}{ll}
0 &\mbox{ if } x\in \Lambda_{\e}:=\displaystyle{\frac{\Lambda}{\e}}, \\
\e^{-1}   &\mbox{ if } x\notin \Lambda_{\e}.
\end{array}
\right.
\end{equation*}
The functional $Q_{\e}$ will act as a penalization to force the concentration phenomena to occur inside $\Lambda$. This type of penalization was first introduced in \cite{BW}.
Clearly, $J_{\e}\in C^{1}(X_{\e}, \R)$ and its differential is given by:
\begin{align*}
\langle J'_{\e}(u), v\rangle&=M(\|u\|^{2}_{\x})\iint_{\R^{N+1}_{+}} y^{1-2s} \nabla u\nabla v\, dx dy+\int_{\R^{N}} V_{\e}(x) u(x, 0) v(x, 0)\, dx\\
&\quad -\int_{\R^{N}} g_{\e}(x, u(x, 0)) v(x, 0)\, dx+4\left( \int_{\R^{N}} \chi_{\e}(x) u^{2}(x,0)\, dx-1 \right)_{+}^{2} \int_{\R^{N}} \chi_{\e}(x) u(x,0) v(x,0)\, dx
\end{align*}
for all $u, v\in X_{\e}$. We stress that a critical point of $P_{\e}$ is a weak  solution to \eqref{MEP}.
In order to find solutions concentrating in $\Lambda$ as $\e\ri 0$, we look for critical points of $J_{\e}$ for which $Q_{\e}$ is zero.

Let $\displaystyle{\delta:=\frac{1}{10}\dist\{\M, \R^{N}\setminus \Lambda\}}$. By $(f_3)$ we can choose $\beta\in (0, \delta)$ sufficiently small such that
\begin{equation}\label{4.3G}
F(T)>\frac{V(x)}{2} T^{2} \quad \mbox{ for all } x\in \M^{5\beta},
\end{equation}
where 
$$
\M^{\beta}:=\{z\in \R^{N}: \inf_{w\in \M} |z-w|\leq \beta\}.
$$
Define a nonincreasing function $\phi_{0}\in C^{\infty}(\R_{+})$ such that $0\leq \phi\leq 1$, $\phi_{0}=1$ in $[0, 1]$, $\phi_{0}=0$ in $[2, \infty)$ and $|\phi'_{0}|_{\infty}\leq C$.
In what follows, we look for solutions to \eqref{MEP} near the set
$$
E_{\e}:=\Bigl\{\phi_{0}(\sqrt{|\e x-x'|^{2}+\e^{2}y^{2}}/\beta) W(x-(x'/\e), y): x'\in \M^{\beta}, W\in \mathcal{S}_{V_{0}}\Bigr\}.
$$
Fix $W^{*}\in \mathcal{S}_{V_{0}}$ and define for $t>0$ and $(x, y)\in \R^{N+1}_{+}$
$$
W_{\e,t}(x,y):=\phi_{0}\left(\frac{\e}{\beta}\sqrt{|x|^{2}+y^{2}}\right) W^{*}\left(\frac{x}{t}, \frac{y}{t}\right).
$$ 
Next we show that $J_{\e}$ has a mountain pass geometry \cite{AR}.
Indeed, by $(M1)$, $(V_1)$, $(f_1)$, $(f_2)$ and ${\rm Tr}(X_{\e})\subset L^{q}(\R^{N})$ for all $q\in [2, \2]$, we have
\begin{align*}
J_{\e}(u)&\geq \frac{m_{0}}{2} \|u\|^{2}_{X^{s}(\R^{N+1}_{+})}+\frac{1}{2}\int_{\R^{N}} V_{\e}(x) u^{2}(x,0)\, dx-\e |u(\cdot, 0)|_{2}^{2}-C_{\e} |u(\cdot, 0)|_{\2}^{\2} \\
&\geq c_{1} \|u\|_{\e}^{2}-c_{2}  \|u\|_{\e}^{\2}.
\end{align*} 
Hence, there exist $\rho, \delta>0$ such that $J_{\e}(u)\geq \delta$ for $\|u\|_{\e}=\rho$. \\
On the other hand, using the fact that $W^{*}$ satisfies the Pohozaev identity and $(M3)$, we have
\begin{align*}
&L_{V_{0}}\left(W^{*}\left(\frac{\cdot}{t}, \frac{\cdot}{t}\right)\right)\\
&=\frac{t^{N}}{2}\left[\frac{1}{t^{N}}\wh{M}\left(t^{N-2s} \|W^{*}\|^{2}_{X^{s}(\R^{N+1}_{+})}\right)- \left(\frac{N-2s}{N} \right) M\left( \|W^{*}\|^{2}_{X^{s}(\R^{N+1}_{+})}\right)  \|W^{*}\|^{2}_{X^{s}(\R^{N+1}_{+})}  \right]\ri -\infty,
\end{align*}
as $t\ri \infty$. Then there exists $t_{0}>0$ such that 
\begin{align}\label{4.2G}
L_{V_{0}}\left(W^{*}\left(\frac{\cdot}{t}, \frac{\cdot}{t}\right)\right)<-2 \quad \forall t\geq t_{0}.
\end{align}

\noindent
Now we prove the following result:
\begin{lem}\label{lem4.1G}
It holds
$$
\sup_{t\in [0, t_{0}]} |J_{\e}(W_{\e, t})-L_{V_{0}}(W^{*}_{t})|\ri 0 \mbox{ as } \e\ri 0,
$$
where $W^{*}_{t}(x,y):=W^{*}(\frac{x}{t}, \frac{y}{t})$ for $t>0$, and $W^{*}_{0}\equiv W_{\e,0}\equiv 0$.
\end{lem}
\begin{proof}
Since $supp(W_{\e, t}(\cdot, 0))\subset \Lambda_{\e}$ and $supp(\chi_{\e})\subset \R^{N}\setminus \Lambda_{\e}$, we have $Q(W_{\e, t})=0$ and $G_{\e}(x, W_{\e, t}(x,0))=F(W_{\e, t}(x,0))$ for all $\e, t\geq 0$ and $x\in \R^{N}$.
Hence, for all $t\in (0, t_{0}]$
\begin{align*}
|J_{\e}(W_{\e, t})-L_{V_{0}}(W^{*}_{t})|&\leq \frac{1}{2} |\wh{M}(\|W_{\e,t}\|^{2}_{\x})-\wh{M}(\|W^{*}_{t}\|^{2}_{\x})|+\frac{1}{2} \int_{\R^{N}} |V_{\e}(x) \phi_{0}(\e|x|/\beta)-V_{0}|(W_{t}^{*}(x,0))^{2}\, dx\\
&+\int_{\R^{N}} |F(W_{\e,t}(x,0))-F(W^{*}_{t}(x,0))|\, dx.
\end{align*}
Note that as $\e\ri 0$
\begin{align}\label{condor}
\|W_{\e, t}\|^{2}_{X^{s}(\R^{N+1}_{+})}=\|W^{*}_{t}\|^{2}_{X^{s}(\R^{N+1}_{+})}+o(1) \mbox{ uniformly in } t\in [0, t_{0}].
\end{align}
Indeed,
\begin{align*}
\|W_{\e, t}\|_{X^{s}(\R^{N+1}_{+})}^{2}&=\iint_{\R^{N+1}_{+}} y^{1-2s} |\nabla \phi_{0}(\e\sqrt{|x|^{2}+y^{2}}/\beta)|^{2}\left(W^{*}\left(\frac{x}{t}, \frac{y}{t}\right)\right)^{2}\, dx dy\\
&+\iint_{\R^{N+1}_{+}} y^{1-2s} |\phi_{0}(\e\sqrt{|x|^{2}+y^{2}}/\beta)|^{2}\left|\nabla W^{*}\left(\frac{x}{t}, \frac{y}{t}\right)\right|^{2}\, dx dy \\
&+2 \iint_{\R^{N+1}_{+}} y^{1-2s} \nabla \phi_{0}(\e\sqrt{|x|^{2}+y^{2}}/\beta) \nabla W^{*}\left(\frac{x}{t}, \frac{y}{t}\right) \phi_{0}(\e\sqrt{|x|^{2}+y^{2}}/\beta) W^{*}\left(\frac{x}{t}, \frac{y}{t}\right) \, dx dy\\
&=:A_{\e,t}+B_{\e,t}+C_{\e,t}.
\end{align*}
Now, by Lemma \ref{lem2.1}, for any $t\in (0, t_{0}]$ we have
\begin{align}\label{3.1}
A_{\e,t}&\leq C\e^{2} \iint_{B^{+}_{\frac{2\beta}{\e}}(0, 0)\setminus B^{+}_{\frac{\beta}{\e}}(0, 0)} y^{1-2s} \left(W^{*}\left(\frac{x}{t}, \frac{y}{t}\right)\right)^{2}\, dx dy \nonumber \\
&\leq C\e^{2} \left[ \iint_{B^{+}_{\frac{2\beta}{\e}}(0, 0)\setminus B^{+}_{\frac{\beta}{\e}}(0, 0)} y^{1-2s} \left(W^{*}\left(\frac{x}{t}, \frac{y}{t}\right)\right)^{2\gamma}\, dx dy  \right]^{\frac{1}{\gamma}}  \left[\iint_{B^{+}_{\frac{2\beta}{\e}}(0, 0)\setminus B^{+}_{\frac{\beta}{\e}}(0, 0)} y^{1-2s}\, dx dy \right]^{1-\frac{1}{\gamma}} \nonumber \\
&\leq C\e^{2} \left[ \iint_{B^{+}_{\frac{2\beta}{\e}}(0, 0)\setminus B^{+}_{\frac{\beta}{\e}}(0, 0)} y^{1-2s} \left(W^{*}\left(\frac{x}{t}, \frac{y}{t}\right)\right)^{2\gamma}\, dx dy  \right]^{\frac{1}{\gamma}} \left[\int_{\frac{\beta}{\e}}^{\frac{2\beta}{\e}} r^{N+1-2s}\, dr \right]^{1-\frac{1}{\gamma}} \nonumber \\
&\leq C \left[ \iint_{B^{+}_{\frac{2\beta}{\e}}(0, 0)\setminus B^{+}_{\frac{\beta}{\e}}(0, 0)} y^{1-2s} \left(W^{*}\left(\frac{x}{t}, \frac{y}{t}\right)\right)^{2\gamma}\, dx dy  \right]^{\frac{1}{\gamma}} \nonumber\\
&\leq C \left[ \iint_{B^{+}_{\frac{2\beta}{t\e}}(0, 0)\setminus B^{+}_{\frac{\beta}{t\e}}(0, 0)}  t^{N+2-2s} y^{1-2s} \left(W^{*}\left(x,y\right)\right)^{2\gamma}\, dx dy  \right]^{\frac{1}{\gamma}} \nonumber\\
&\leq C \left[ \iint_{\R^{N+1}_{+}\setminus B^{+}_{\frac{\beta}{t_{0}\e}}(0, 0)}  t^{N+2-2s}_{0} y^{1-2s} \left(W^{*}\left(x,y\right)\right)^{2\gamma}\, dx dy  \right]^{\frac{1}{\gamma}}\ri 0 \mbox{ as } \e\ri 0.
\end{align}
On the other hand, for $t\in (0, t_{0}]$, using that $0\leq \phi_{0}\leq 1$ and $\phi_{0}$ is nonincreasing we get
\begin{align*}
&\left|B_{\e,t}-\iint_{\R^{N+1}_{+}} y^{1-2s} |\nabla W^{*}\left(\frac{x}{t}, \frac{y}{t}\right)|^{2}\, dx dy\right| \\
&\leq \iint_{\R^{N+1}_{+}} y^{1-2s} [1-(\phi_{0}(\e\sqrt{|x|^{2}+y^{2}}/\beta))^{2}] \left|\nabla W^{*}\left(\frac{x}{t}, \frac{y}{t}\right)\right|^{2}\, dx dy\\
&=\iint_{\R^{N+1}_{+}} t^{N-2s} y^{1-2s} [1-(\phi_{0}(\e t\sqrt{|x|^{2}+y^{2}}/\beta))^{2}] \left|\nabla W^{*}\left(x, y\right)\right|^{2}\, dx dy \\
&\leq \iint_{\R^{N+1}_{+}} t^{N-2s}_{0} y^{1-2s} [1-(\phi_{0}(\e t_{0}\sqrt{|x|^{2}+y^{2}}/\beta))^{2}] \left|\nabla W^{*}\left(x, y\right)\right|^{2}\, dx dy \ri 0 \mbox{ as } \e\ri 0.
\end{align*}
Since H\"older's inequality yields $C_{\e,t}\leq A_{\e,t}^{1/2} B_{\e,t}^{1/2}$, we deduce that 
$$
\sup_{t\in [0, t_{0}]} C_{\e,t}\ri 0 \mbox{ as } \e\ri 0.
$$ 
Therefore \eqref{condor} holds true. 

Now, noting that $\|W_{\e, t}\|^{2}_{X^{s}(\R^{N+1}_{+})}, \|W^{*}_{t}\|^{2}_{X^{s}(\R^{N+1}_{+})}\leq C$ for all $t\in [0, t_{0}]$ and $\e>0$ sufficiently small, and using $\wh{M}(t_{2})- \wh{M}(t_{1})= \int_{t_{1}}^{t_{2}} M(\tau)\, d\tau$ and $(M4)$, we see that
$$
\left|\wh{M}\left(\|W_{\e, t}\|^{2}_{X^{s}(\R^{N+1}_{+})}\right)-\wh{M}\left(\|W^{*}_{t}\|^{2}_{X^{s}(\R^{N+1}_{+})}\right) \right|\leq M(C) \left|\|W_{\e, t}\|^{2}_{X^{s}(\R^{N+1}_{+})}- \|W^{*}_{t}\|^{2}_{X^{s}(\R^{N+1}_{+})} \right|
$$
which together with \eqref{condor} implies that 
$$
\wh{M}\left(\|W_{\e, t}\|^{2}_{X^{s}(\R^{N+1}_{+})}\right)= \wh{M}\left(\|W^{*}_{t}\|^{2}_{X^{s}(\R^{N+1}_{+})}\right)+o(1) \mbox{ uniformly in } t\in [0, t_{0}].
$$
On the other hand, recalling that (see \cite{FQT}) $W^{*}(\cdot, 0)$ has the following polynomial type-decay
$$
0<W^{*}(x,0)\leq \frac{C}{1+|x|^{N+2s}} \quad \forall x\in \R^{N},
$$
we have  
\begin{align}\label{GROWTH}
0<W^{*}_{t}(x,0)\leq \frac{Ct^{N+2s}_{0}}{t^{N+2s}_{0}+|x|^{N+2s}} \quad \forall x\in \R^{N}, t\in (0, t_{0}],
\end{align}
which together with $0\leq V_{\e}(x) \phi_{0}(\e |x|/\beta)\leq \max_{x\in \overline{\Gamma^{0}_{2\beta}}(0)} V(x) $ and $\phi_{0}(\e \cdot)\ri 1$ as $\e\ri 0$, implies that
$$
\lim_{\e\ri 0} \sup_{t\in [0, t_{0}]} \left|\int_{\R^{N}} [V_{\e}(x) \phi_{0}(\e|x|/\beta)-V_{0}](W_{t}^{*}(x,0))^{2}\, dx \right|=0.
$$
Finally, observing that
$$
F(a+b)-F(a)=b\int_{0}^{1} f(a+\tau b)\, d\tau,
$$
it follows from $(f_1)$ and $(f_2)$ that
\begin{align*}
&\int_{\R^{N}} |F(W_{\e,t}(x,0))-F(W^{*}_{t}(x,0))|\, dx\\
&\leq \int_{\R^{N}} |W_{\e, t}(x,0)-W^{*}_{t}(x,0)| \int_{0}^{1} |f(W^{*}_{t}(x,0)+\tau (W_{\e, t}(x,0)-W^{*}_{t}(x,0))|\, d\tau dx\\
&\leq C\int_{\R^{N}} |W_{\e, t}(x,0)-W^{*}_{t}(x,0)| [|W^{*}_{t}(x,0)|+|W_{\e, t}(x,0)-W^{*}_{t}(x,0)|\\
&\quad+|W^{*}_{t}(x,0)|^{\2-1}+|W_{\e, t}(x,0)-W^{*}_{t}(x,0)|^{\2-1}]\, dx.
\end{align*}
Taking into account $W_{\e,t}(x,0)-W^{*}_{t}(x,0)=(\phi_{0}(\e |x|/\beta)-1)W^{*}_{t}(x,0)$, \eqref{GROWTH} and $\phi_{0}(\e \cdot)\ri 1$ as $\e\ri 0$, we get
$$
\lim_{\e\ri 0} \sup_{t\in [0, t_{0}]} \left|\int_{\R^{N}} F(W_{\e,t}(x,0))-F(W^{*}_{t}(x,0))\, dx\right|=0.
$$
\end{proof}

\noindent
Notice that from \eqref{4.2G} and Lemma \ref{lem4.1G} there exists $\e_{0}$ sufficiently small such that
\begin{align*}
|J_{\e}(W_{\e, t_{0}})-L_{V_{0}}(W^{*}_{t_{0}})|\leq -L_{V_{0}}(W_{t_{0}})-2 \quad J_{\e}(W_{\e, t_{0}})<-2 \quad \mbox{for} \e\in (0, \e_{0}).
\end{align*}
Therefore, we can define the minimax level
$$
c_{\e}:=\inf_{\gamma\in \Gamma_{\e}} \max_{t\in [0, 1]} J_{\e}(\gamma(t))
$$
where
$$
\Gamma_{\e}:=\{\gamma\in C([0, 1], X_{\e}): \gamma(0)=0, \gamma(1)=W_{\e, t_{0}}\}.
$$

\begin{lem}\label{prop4.1G}
$\lim_{\e\ri 0} c_{\e}=c_{V_{0}}$.
\end{lem}
\begin{proof}
We first prove that 
\begin{equation}\label{4.5G}
\limsup_{\e\ri 0}c_{\e}\leq c_{V_{0}}.
\end{equation}
Since $W_{\e, t}\ri 0$ in $X_{\e}$ as $t\ri 0$, and setting 
\begin{equation}\label{4.6G}
\gamma_{\e}(\tau):=W_{\e, \tau t_{0}} \mbox{ for } \tau\in (0, 1], \, \gamma_{\e}(0)=0,
\end{equation}
we see that $\gamma_{\e}\in \Gamma_{\e}$ and thus
\begin{equation}\label{4.7G}
c_{\e}\leq \max_{t\in [0, 1]} J_{\e}(\gamma_{\e}(t))= \max_{t\in [0, t_{0}]} J_{\e}(W_{\e, t}).
\end{equation} 
By Lemma \ref{lem4.1G}, Pohozaev Identity and $(M5)$ we deduce that
\begin{align*}
\max_{t\in [0, t_{0}]} J_{\e}(W_{\e, t})&=\max_{t\in [0, t_{0}]} L_{V_{0}}\left(W^{*}\left(\frac{\cdot}{t}, \frac{\cdot}{t}\right)\right)+o(1)\\
&=\max_{t\in [0, t_{0}]} \left[\frac{1}{2} \wh{M}\left(t^{N-2s} \|W^{*}\|^{2}_{X^{s}(\R^{N+1}_{+})}\right) -t^{N} \left(\frac{N-2s}{2N} \right) M\left(\|W^{*}\|^{2}_{X^{s}(\R^{N+1}_{+})}\right) \|W^{*}\|^{2}_{X^{s}(\R^{N+1}_{+})}\right]+o(1) \\
&\leq L_{V_{0}}(W^{*})+o(1)=c_{V_{0}}+o(1).
\end{align*}
Next, we show that 
\begin{equation}\label{4.9G}
\liminf_{\e\ri 0}c_{\e}\geq c_{V_{0}}.
\end{equation}
Assume by contradiction that $\liminf_{\e\ri 0}c_{\e}< c_{V_{0}}$. Then there exist $\alpha>0$, $\e_{n}\ri 0$ and $\gamma_{n}\in \Gamma_{\e_{n}}$ such that $\max_{t\in [0, 1]} J_{\e_{n}}(\gamma_{n}(t))<c_{V_{0}}-\alpha$. Take $\e_{n}$ such that
$$
\frac{V_{0}}{2} \e_{n}[1+(1+c_{V_{0}})^{2}]<\min\{\alpha, 1\} \mbox{ and } P_{\e_{n}}(\gamma_{n}(1))<-2.
$$ 
Denoting $\e_{n}$ by $\e$ and $\gamma_{n}$ by $\gamma$, since $P_{\e}(\gamma(0))=0$, we can find $t_{0}\in (0, 1)$ such that 
$$
P_{\e}(\gamma(t_{0}))=-1 \mbox{ and } P_{\e}(\gamma(t)) \quad \forall t\in [0, t_{0}].
$$
Hence,
$$
Q_{\e}(\gamma(t))\leq J_{\e}(\gamma(t))+1<c_{V_{0}}-\alpha+1<c_{V_{0}}+1
$$
and consequently
$$
\int_{\R^{N}\setminus \Lambda_{\e}} \gamma(t)^{2}\, dx\leq \e[1+(1+c_{V_{0}})^{2}] \mbox{ for } t\in [0, t_{0}].
$$
Since $G(x,t)\leq F(t)$ we obtain for $t\in [0, t_{0}]$
\begin{align*}
P_{\e}(\gamma(t))&\geq L_{V_{0}}(\gamma(t))-\frac{V_{0}}{2}\int_{\R^{N}\setminus \Lambda_{\e}} \gamma(t)^{2}\, dx\\
&\geq L_{V_{0}}(\gamma(t))-\frac{V_{0}}{2} \e[1+(1+c_{V_{0}})^{2}] 
\end{align*}
which yields
$$
L_{V_{0}}(\gamma(t_{0}))\leq \frac{V_{0}}{2}\e[1+(1+c_{V_{0}})^{2}]-1<0.
$$
On the other hand, the mountain pass level corresponds to the least energy level (see Lemma \ref{Final}), so we have
$$
\max_{t\in [0, t_{0}]} L_{V_{0}}(\gamma(t))\geq c_{V_{0}}.
$$
From
$$
c_{V_{0}}-\alpha>\max_{t\in [0, 1]} L_{V_{0}}(\gamma(t))\geq \max_{t\in [0, t_{0}]} P_{\e}(\gamma(t))
$$
we get
$$
c_{V_{0}}-\alpha>c_{V_{0}}-\frac{V_{0}}{2} \e[1+(1+c_{V_{0}})^{2}]>c_{V_{0}}-\alpha 
$$
and this gives a contradiction.\\
Now, we define
\begin{equation}\label{DE}
d_{\e}:=\max_{t\in [0, 1]} J_{\e}(\gamma_{\e}(t)),
\end{equation}
where $\gamma_{\e}$ is given in \eqref{4.6G}. Then, by \eqref{4.5G}, \eqref{4.7G} and \eqref{4.9G} we see that $c_{\e}\leq d_{\e}$ and
$$
\lim_{\e\ri 0} d_{\e}=\lim_{\e\ri 0} c_{\e}=c_{V_{0}}.
$$
This ends the proof of lemma.
\end{proof}

\noindent
Now we  use the notations
$$
J_{\e}^{b}:=\{w\in X_{\e}: J_{\e}(w)\leq b\},
$$
and for $A\subset X_{\e}$
$$
A^{b}:=\{w\in X_{\e}: \inf_{v\in A} \|w-v\|_{\e}\leq b\}.
$$
The next lemma will be crucial to prove the main result of this work.
\begin{lem}\label{Klem}
There exists $d_{0}>0$ such that for any $(\e_{n})$ and $(w_{\e_{n}})$ with
\begin{align*}
\lim_{n\ri \infty}\e_{n}= 0, \, w_{\e_{n}}\in E_{\e_{n}}^{d_{0}}, \, \lim_{n\ri \infty} J_{\e_{n}}(w_{\e_{n}})\leq c_{V_{0}}, \, \lim_{n\ri \infty} \|J'_{\e_{n}}(w_{\e_{n}})\|_{(X_{\e_{n}})^{-1}}=0,
\end{align*}
there exists, up to a subsequence, $(z_{n})\subset \R^{N}$, $x_{0}\in \M$ and $W\in \mathcal{S}_{V_{0}}$ such that
\begin{align*}
\lim_{n\ri \infty}|\e_{n} z_{n}- x_{0}|=0 \mbox{ and } \lim_{n\ri \infty} \|w_{\e_{n}} -\phi_{0}(\e_{n}\sqrt{|x-z_{n}|^{2}+y^{2}}/\beta) W(x-z_{n}, y)\|_{\e_{n}}=0.
\end{align*}
\end{lem}
\begin{proof}
For simplicity, we write $\e$ instead of $\e_{n}$ and the same will be done for the subsequences.
By the definition of $E_{\e}^{d_{0}}$ and the compactness of $\mathcal{S}_{V_{0}}$ and $\M^{\beta}$, there exist $W_{0}\in \mathcal{S}_{V_{0}}$ and $(x_{\e})\subset \M^{\beta}$ such that for all $\e>0$ small enough 
\begin{align}\label{3.2}
\left\|w_{\e} -\phi_{0}\left(\frac{\e}{\beta}\sqrt{\left|x-\frac{x_{\e}}{\e}\right|^{2}+y^{2}}\right) W_{0}\left(x-\frac{x_{\e}}{\e}, y\right)\right\|_{\e}\leq 2d_{0},
\end{align}
and, as $\e\ri 0$,
$$
x_{\e}\ri x_{0}\in \M^{\beta}.
$$
In what follows, we prove that there exist $(w_{\e,1}), (w_{\e,2})\subset X_{\e}$, $(k_{\e}), (j_{\e})\subset \mathbb{N}$
such that
\begin{compactenum}[$(i)$]
\item $k_{\e}\leq \sqrt{\beta_{\e}/5\e}$ and $k_{\e}\ri \infty$ as $\e\ri 0$, $0\leq j_{\e}\leq k_{\e}-1$, $|w_{\e,1}|, |w_{\e, 2}|\leq |w_{\e}|$,
\item $w_{\e,1}=w_{\e}$ in $B^{+}_{(\frac{2\beta_{\e}}{\e})+(5j_{\e}+1)k_{\e}}(\frac{x_{\e}}{\e},0)$, \, $w_{\e,2}=w_{\e}$ in $\R^{N+1}_{+}\setminus B^{+}_{(\frac{2\beta_{\e}}{\e})+(5j_{\e}+4)k_{\e}}(\frac{x_{\e}}{\e},0)$
\item $supp(w_{\e,1})\subset \overline{B^{+}_{(\frac{2\beta_{\e}}{\e})+(5j_{\e}+2)k_{\e}}(\frac{x_{\e}}{\e},0)}$, \, $supp(w_{\e,2})\subset \R^{N+1}_{+}\setminus B^{+}_{(\frac{2\beta_{\e}}{\e})+(5j_{\e}+3)k_{\e}}(\frac{x_{\e}}{\e},0)$,
\item $\|w_{\e}-w_{\e,1}-w_{\e,2}\|_{\e}\ri 0$ as $\e\ri 0$,
\item $\|w_{\e}\|_{X^{s}_{0}(B_{j_{\e,\e}})}\ri 0$ and 
$$
\iint_{B_{j_{\e}, \e}} y^{1-2s} |w_{\e}|^{2\gamma}\, dx dy\ri 0 \mbox{ as } \e\ri 0,
$$
where 
$$
B_{j_{\e},\e}:=\overline{B^{+}_{(\frac{2\beta_{\e}}{\e})+5(j_{\e}+1)k_{\e}}\left(\frac{x_{\e}}{\e},0\right)}\setminus B^{+}_{(\frac{2\beta_{\e}}{\e})+5j_{\e} k_{\e}}\left(\frac{x_{\e}}{\e},0\right),
$$
and
$$
\int_{\Gamma_{j_{\e}, \e}} V_{\e}(x) |w_{\e}(x,0)|^{2}\, dx\ri 0 \mbox{ as } \e\ri 0,
$$
where
$$
\Gamma_{j_{\e},\e}:=\overline{\Gamma^{0}_{(\frac{2\beta_{\e}}{\e})+5(j_{\e}+1)k_{\e}}\left(\frac{x_{\e}}{\e}\right)}\setminus \Gamma^{0}_{(\frac{2\beta_{\e}}{\e})+5j_{\e} k_{\e}}\left(\frac{x_{\e}}{\e}\right).
$$
\end{compactenum} 
Let $k_{\e}\in \mathbb{N}$ be such that $k_{\e}\leq \sqrt{\frac{\beta}{5\e}}$ and $k_{\e}\ri \infty$ as $\e\ri 0$, and put $\widetilde{w}_{\e}(x,y):=w_{\e}(x+\frac{x_{\e}}{\e},y)$. By \eqref{3.2}, Lemma \ref{lem2.1}-$(i)$ and $\phi_{0}(\e \sqrt{|x|^{2}+y^{2}}/\beta)=0$ in $\R^{N+1}_{+}\setminus B^{+}_{\frac{2\beta}{\e}}(0,0)$ we have 
\begin{align}\label{3.3}
&\iint_{\R^{N+1}_{+}\setminus B^{+}_{\frac{2\beta}{\e}}(0,0)} y^{1-2s}|\nabla \tilde{w}_{\e}|^{2}\, dx dy+\int_{\R^{N}\setminus \Gamma^{0}_{\frac{2\beta}{\e}}(0)} V(\e x+x_{\e}) |\tilde{w}_{\e}(x,0)|^{2}\, dx \nonumber\\
&\quad +\left(\iint_{\R^{N+1}_{+}\setminus B^{+}_{\frac{2\beta}{\e}}(0,0)} y^{1-2s}|\tilde{w}_{\e}|^{2\gamma}\, dx dy\right)^{\frac{1}{\gamma}}\leq Cd_{0}.
\end{align}
For all $j=0, 1, \dots, k_{\e}-1$, we set 
$$
\tilde{B}_{j, \e}:=\overline{B^{+}_{(\frac{2\beta_{\e}}{\e})+5(j+1)k_{\e}}}(0,0)\setminus B^{+}_{(\frac{2\beta_{\e}}{\e})+5j k_{\e}}(0,0) \mbox{ and } \tilde{\Gamma}_{j, \e}:=\overline{\Gamma^{0}_{(\frac{2\beta_{\e}}{\e})+5(j+1)k_{\e}}}(0)\setminus \Gamma^{0}_{(\frac{2\beta_{\e}}{\e})+5j k_{\e}}(0). 
$$
From \eqref{3.3} we deduce that
\begin{align*}
&\sum_{j=0}^{k_{\e}-1} \iint_{\tilde{B}_{j, \e}} y^{1-2s}|\nabla \tilde{w}_{\e}|^{2}\, dx dy+\sum_{j=0}^{k_{\e}-1} \int_{\tilde{\Gamma}_{j, \e}} V(\e x+x_{\e}) |\tilde{w}_{\e}(x,0)|^{2}\, dx\\
&\quad+\sum_{j=0}^{k_{\e}-1} \left(\iint_{\tilde{B}_{j, \e}} y^{1-2s}|\tilde{w}_{\e}|^{2\gamma}\, dx dy\right)^{\frac{1}{\gamma}}\leq Cd_{0}.
\end{align*}
Hence, there exists $j_{\e}\in\{0, 1, \dots, k_{\e}-1\}$ such that
\begin{align}\label{3.4}
&\iint_{\tilde{B}_{j_{\e}, \e}} y^{1-2s}|\nabla \tilde{w}_{\e}|^{2}\, dx dy+ \int_{\tilde{\Gamma}_{j_{\e}, \e}} V(\e x+x_{\e}) |\tilde{w}_{\e}(x,0)|^{2}\, dx \nonumber\\
&\quad+\left(\iint_{\tilde{B}_{j_{\e}, \e}} y^{1-2s}|\tilde{w}_{\e}|^{2\gamma}\, dx dy\right)^{\frac{1}{\gamma}}\leq Cd_{0}/k_{\e}\ri 0 \mbox{ as } \e\ri 0.
\end{align}
Define two cut-off functions $(\xi_{\e, 1})$ and $(\xi_{\e,2})$ such that
\begin{equation*}
\xi_{\e,1}:=\left\{
\begin{array}{ll}
1 &\mbox{ in } \overline{B^{+}_{\frac{2\beta}{\e}+(5j_{\e}+1)k_{\e}}(0,0)}, \\
\medskip
0  &\mbox{ in } \R^{N+1}_{+}\setminus B^{+}_{\frac{2\beta}{\e}+(5j_{\e}+2)k_{\e}}(0,0),
\end{array}
\right.
\end{equation*}
and
\begin{equation*}
\xi_{\e,2}:=\left\{
\begin{array}{ll}
0 &\mbox{ in } \overline{B^{+}_{\frac{2\beta}{\e}+(5j_{\e}+3)k_{\e}}(0,0)}, \\
\medskip
1  &\mbox{ in } \R^{N+1}_{+}\setminus B^{+}_{\frac{2\beta}{\e}}+(5j_{\e}+4)k_{\e}(0,0),
\end{array}
\right.
\end{equation*}
and $0\leq \xi_{\e,1}, \xi_{\e,2}\leq 1$, $|\nabla \xi_{\e,1}|, |\nabla \xi_{\e,2}|\leq \frac{C}{k_{\e}}$,  
and we set
$$
\tilde{w}_{\e,i}:=\xi_{\e, i} \tilde{w}_{\e} \mbox{ and } w_{\e,i}(x,y):=\tilde{w}_{\e,i}\left(x-\frac{x_{\e}}{\e}, y\right) \mbox{ for } i=1,2.
$$
Since $w_{\e}\in X_{\e}$, we see that $w_{\e,i}\in X_{\e_{i}}$ for $i=1, 2$. Hence, $(i)$-$(iii)$ hold true.
Now, direct calculations show that
\begin{align*}
&\|w_{\e}-w_{\e,1}-w_{\e,2}\|^{2}_{\e}\leq C \iint_{B^{+}_{\frac{2\beta}{\e}+(5j_{\e}+4)k_{\e}}(0,0)\setminus B^{+}_{\frac{2\beta}{\e}+(5j_{\e}+1)k_{\e}}(0,0)}  y^{1-2s} |\nabla \tilde{w}_{\e}|^{2}\, dx dy \\
&+C \int_{\Gamma^{0}_{\frac{2\beta}{\e}+(5j_{\e}+4)k_{\e}}(0)\setminus \Gamma^{0}_{\frac{2\beta}{\e}+(5j_{\e}+1)k_{\e}}(0)}  V(\e x+x_{\e}) |\tilde{w}_{\e}|^{2}\, dx \\
&+C \iint_{B^{+}_{\frac{2\beta}{\e}+(5j_{\e}+2)k_{\e}}(0,0)\setminus B^{+}_{\frac{2\beta}{\e}+(5j_{\e}+1)k_{\e}}(0,0)}  y^{1-2s} |\nabla \xi_{\e,1}|^{2} |\tilde{w}_{\e}|^{2} \, dx dy \\
&+C \iint_{B^{+}_{\frac{2\beta}{\e}+(5j_{\e}+4)k_{\e}}(0,0)\setminus B^{+}_{\frac{2\beta}{\e}+(5j_{\e}+3)k_{\e}}(0,0)}  y^{1-2s} |\nabla \xi_{\e,2}|^{2} |\tilde{w}_{\e}|^{2} \, dx dy\\
&=:(I)_{\e}+(II)_{\e}+(III)_{\e}+(IV)_{\e}.
\end{align*}
Using \eqref{3.4} we deduce that $(I)_{\e}, (II)_{\e}=o(1)$. Moreover, arguing as in \eqref{3.1}, it follows from \eqref{3.4} that
\begin{align*}
(III)_{\e}\leq C\left(\iint_{B^{+}_{\frac{2\beta}{\e}+(5j_{\e}+2)k_{\e}}(0,0)\setminus B^{+}_{\frac{2\beta}{\e}+(5j_{\e}+1)k_{\e}}(0,0)} y^{1-2s} |\tilde{w}_{\e}|^{2\gamma} \, dx dy  \right)^{\frac{1}{\gamma}}=o(1).
\end{align*}
In a similar fashion we can prove that $(IV)_{\e}=o(1)$. In conclusion, $(iv)$ holds true. Moreover, by \eqref{3.4}, we see that $(v)$ is  satisfied.
Taking into account $(i)$-$(v)$, $(f_1)$-$(f_2)$ and the boundedness of $(w_{\e})$ in $X_{\e}$ we get
\begin{align}
&\|w_{\e}\|^{2}_{X^{s}(\R^{N+1}_{+})}=\|w_{\e,1}\|^{2}_{X^{s}(\R^{N+1}_{+})}+\|w_{\e,2}\|^{2}_{X^{s}(\R^{N+1}_{+})}+o(1), \label{G1}\\
&\int_{\R^{N}} V_{\e}(x) w_{\e}^{2}(x,0)\, dx=\int_{\R^{N}} V_{\e}(x) w_{\e,1}^{2}(x,0)\, dx+\int_{\R^{N}} V_{\e}(x) w_{\e,2}^{2}(x,0)\, dx+o(1), \label{G2}\\
&\int_{\R^{N}} F(w_{\e}(x,0))\, dx=\int_{\R^{N}} F(w_{\e,1}(x,0))\, dx+\int_{\R^{N}} F(w_{\e,2}(x,0))\, dx+o(1). \label{G3}
\end{align}
By $(M1)$, we know that
$$
\wh{M}(t_{1}+t_{2})= \wh{M}(t_{1})+\int_{t_{1}}^{t_{1}+t_{2}} M(\tau)\, d\tau\geq  \wh{M}(t_{1})+m_{0} t_{2},
$$
which together with \eqref{G1}-\eqref{G3}, the boundedness of $(w_{\e})$ in $X_{\e}$ and $G(x,t)\leq F(t)$ implies that
\begin{align}\label{3.5}
J_{\e}(w_{\e})\geq I_{\e}(w_{\e,1})+\frac{m_{0}}{2}\|w_{\e,2}\|^{2}_{X^{s}(\R^{N+1}_{+})} +\frac{1}{2}\int_{\R^{N}} V_{\e}(x) w_{\e,2}^{2}(x,0)\, dx-\int_{\R^{N}} F(w_{\e,2}(x,0))\, dx+o(1).
\end{align}
Now, we prove that $\|w_{\e,2}\|_{\e}\ri 0$ as $\e\ri 0$. By \eqref{3.2}, $(iv)$ and the definition of $w_{\e,2}$, we see that
\begin{align*}
\|w_{\e,2}\|_{\e}&\leq \left\|w_{\e,1}- \phi_{0}\left(\frac{\e}{\beta}\sqrt{\left|x-\frac{x_{\e}}{\e}\right|^{2}+y^{2}}\right) W_{0}\left(x-\frac{x_{\e}}{\e}, y\right)\right\|_{\e}+2d_{0}+o(1)\\
&=\left\| w_{\e,1}- \phi_{0}\left(\frac{\e}{\beta}\sqrt{\left|x-\frac{x_{\e}}{\e}\right|^{2}+y^{2}}\right) W_{0}\left(x-\frac{x_{\e}}{\e}, y\right)\right\|_{X_{\e}\left(B^{+}_{\frac{2\beta}{\e}+(5j_{\e}+2)k_{\e}}(0,0)\right)} +2d_{0}+o(1)\\
&\leq \|w_{\e,2}\|_{X_{\e}\left(B^{+}_{\frac{2\beta}{\e}+(5j_{\e}+2)k_{\e}}(0,0)\right)}+2d_{0}+o(1) \\
&=4d_{0}+o(1),
\end{align*}
which yields
\begin{align}\label{GRIGN}
\limsup_{\e\ri 0} \|w_{\e,2}\|_{\e}\leq 4d_{0}.
\end{align}
On the other hand, using $\langle J'_{\e}(w_{\e}), w_{\e,1}\rangle=o(1)$, $\langle Q'_{\e}(w_{\e}), w_{\e,2}\rangle=\langle Q'_{\e}(w_{\e,2}), w_{\e,2}\rangle\geq 0$, $(M1)$, $(V_1)$, $(f_1)$-$(f_2)$, $(iii)$, $(iv)$, \eqref{GRIGN}, the boundedness of $(w_{\e})$ in $X_{\e}$, we get
\begin{align*}
&m_{0} \iint_{\R^{N+1}_{+}} y^{1-2s} |\nabla w_{\e,2}|^{2} \, dx dy+\int_{\R^{N}} V_{\e}(x) w_{\e,2}^{2}(x,0)\, dx\\
&\leq M(\|w_{\e}\|^{2}_{\e}) \iint_{\R^{N+1}_{+}} y^{1-2s} |\nabla w_{\e,2}|^{2} \, dx dy+\int_{\R^{N}} V_{\e}(x) w_{\e,2}^{2}(x,0)\, dx\\
&\leq M(\|w_{\e}\|^{2}_{\e}) \iint_{\R^{N+1}_{+}} y^{1-2s} |\nabla w_{\e,2}|^{2} \, dx dy+\int_{\R^{N}} V_{\e}(x) w_{\e,2}^{2}(x,0)\, dx+\langle Q'(w_{\e,2}), w_{\e,2}\rangle\\
&= \int_{\R^{N}} g_{\e}(x, w_{\e,2}(x,0)) w_{\e,2}(x,0)\, dx+o(1)\\
&\leq \delta \int_{\R^{N}} w_{\e,2}^{2}(x,0)\, dx+C_{\delta} \int_{\R^{N}} |w_{\e,2}(x,0)|^{\2}+o(1) \\
&\leq \frac{\delta}{V_{1}} \int_{\R^{N}} V_{\e}(x) w_{\e,2}^{2}(x,0)\, dx+C_{\delta} |w_{\e,2}(x,0)|_{\2}^{\2}+o(1).
\end{align*}
Then, choosing $\delta>0$ sufficiently small and using Lemma \ref{Sobolev} we deduce that $\|w_{\e,2}\|^{2}_{\e}\leq C\|w_{\e,2}\|^{\2}_{\e}+o(1)$. Taking $d_{0}>0$ small enough, we deduce that $\|w_{\e,2}\|_{\e}=o(1)$. Hence, in view of \eqref{3.5}, we have
\begin{align}\label{3.6}
J_{\e}(w_{\e})\geq I_{\e}(w_{\e,1})+o(1).
\end{align}
Up to a subsequence, we can find $\tilde{w}\in X^{1,s}(\R^{N+1}_{+})$ such that
\begin{align}\label{3.8}
\tilde{w}_{\e,1}\rightharpoonup \tilde{w} \mbox{ in }  X^{1,s}(\R^{N+1}_{+}) \mbox{ and } \tilde{w}_{\e,1}(\cdot, 0)\rightharpoonup \tilde{w}(\cdot, 0) \mbox{ in }  L^{q}_{loc}(\R^{N}) \quad \forall q\in [1, \2).
\end{align}
In what follows we show that
\begin{align}\label{3.25}
\tilde{w}_{\e,1}(\cdot, 0)\rightarrow \tilde{w}(\cdot, 0) \mbox{ in }  L^{q}(\R^{N}) \quad \forall q\in (2, \2).
\end{align}
Indeed, by vanishing Lions-type lemma (see Lemma 3.3 in \cite{HZ}), we assume by contradiction that there exists $r>0$ such that
\begin{align*}
\liminf_{\e\ri 0} \sup_{z\in \R^{N}} \int_{\Gamma^{0}_{1}(z)} |\tilde{w}_{\e,1}(x,0)-\tilde{w}(x,0)|^{2}\, dx=2r>0.
\end{align*}
Then, for $\e>0$ small, there exists $z_{\e}\in \R^{N}$ such that
\begin{align}\label{3.9}
\int_{\Gamma^{0}_{1}(z_{\e})} |\tilde{w}_{\e,1}(x,0)-\tilde{w}(x,0)|^{2}\, dx\geq r>0.
\end{align}
By \eqref{3.8} we see that $(z_{\e})$ is unbounded, so, up to a subsequence, $|z_{\e}|\ri \infty$. 
Then, by \eqref{3.9}, 
\begin{align}\label{3.99}
\liminf_{\e\ri 0} \int_{\Gamma^{0}_{1}(z_{\e})} |\tilde{w}_{\e,1}(x,0)|^{2}\, dx\geq r>0.
\end{align}
Since $\xi_{\e,1}(x,0)=0$ for $|x|\geq (\frac{2\beta}{\e})+(5j_{\e}+2)k_{\e}$, we deduce that $|z_{\e}|<(\frac{2\beta}{\e})+(5j_{\e}+3)k_{\e}$ for $\e>0$ small enough. Therefore, we may assume that
\begin{align}\label{CONV}
\e z_{\e}\ri z_{0}\in \overline{\Gamma_{3\beta}^{0}}(0) \quad \mbox{ and } \quad \bar{w}_{\e}(x, y):= \tilde{w}_{\e, 1}(x+ z_{\e}, y)\rightharpoonup \bar{w}(x, y) \, \mbox{ in } X^{1, s}(\R^{N+1}_{+}).
\end{align}
Now, we show that $\bar{w}$ satisfies
\begin{align}\label{Pbarw}
\left\{
\begin{array}{ll}
-\dive(y^{1-2s} \nabla \bar{w})=0 &\mbox{ in } \R^{N+1}_{+}, \\
\frac{1}{\alpha_{0}} \frac{\partial \bar{w}}{\partial \nu^{1-2s}}=-V(x_{0}+z_{0}) \bar{w}(\cdot,0)+f(\bar{w}(\cdot,0)) &\mbox{ in } \R^{N},
\end{array}
\right.
\end{align}
where
$$
\alpha_{0}:=\lim_{\e\ri 0} M(\|w_{\e}\|^{2}_{\x}).
$$
Fix $k\geq 1$. Since $x_{0}+z_{0}\in \M^{4\beta}\subset \Lambda$, there exists $n_{0}=n_{0}(k)\in \mathbb{N}$ such that $\e x+x_{\e}+\e z_{\e}\in \Lambda$ for all $x\in \Gamma^{0}_{k}(0)$ and $n\geq n_{0}$. By the definition of $\chi_{\e}$ and $g(x,t)$ it follows that
$$
\left\langle Q'(w_{\e}),\phi\left(\cdot-\frac{x_{\e}}{\e}-z_{\e}\right)\right\rangle=0 \mbox{ and } g(\e x+x_{\e}+\e z_{\e}, t)\phi=f(t)\phi,
$$
for all $n\geq n_{0}$ and $\phi\in C^{\infty}_{c}(B^{+}_{k}(0, 0)\cup \Gamma^{0}_{k}(0))$. From $\langle J'_{\e}(w_{\e}), \phi(\cdot-\frac{x_{\e}}{\e}-z_{\e})\rangle=o(1)$, $(iv)$ and $\|w_{\e,2}\|_{\e}=o(1)$ we can deduce that
\begin{align*}
o(1)&=M(\|w_{\e}\|^{2}_{X^{s}(\R^{N+1}_{+})})\iint_{\R^{N+1}_{+}} y^{1-2s} \nabla \bar{w}_{\e} \nabla \phi\, dx dy\\
&\quad+\int_{\R^{N}} V(\e x+x_{\e}+\e z_{\e}) \bar{w}_{\e}(x,0) \phi(x,0)\, dx-\int_{\R^{N}} f(\bar{w}_{\e}(x,0)) \phi(x,0)\, dx. 
\end{align*}
Note that by $(M1)$ and the boundedness of $(w_{\e})$ in $X_{\e}$ it holds $m_{0}\leq \alpha_{0}\leq C$. Then, by \eqref{CONV} and the arbitrariness of $k$ we get
\begin{align*}
0=\alpha_{0}\iint_{\R^{N+1}_{+}} y^{1-2s} \nabla \bar{w} \nabla \phi\, dx dy+\int_{\R^{N}} V(x_{0}+z_{0}) \bar{w}(x,0) \phi(x,0)\, dx-\int_{\R^{N}} f(\bar{w}(x,0)) \phi(x,0)\, dx,
\end{align*}
for all $\phi\in C^{\infty}_{c}(\R^{N+1}_{+})$, which proves the claim.

Since $\bar{w}\neq 0$ by \eqref{3.99}, it follows from the Pohozaev identity that
\begin{align}\label{3.999}
d_{V(x_{0}+z_{0})}\leq \frac{s}{N}\alpha_{0} \int_{\R^{N+1}_{+}} y^{1-2s} |\nabla \bar{w}|^{2}\, dx dy,
\end{align}
where
$$
d_{V(x_{0}+z_{0})}:=\inf\left\{L_{\alpha_{0}, V(x_{0}+z_{0})}(u): u\in X^{1,s}(\R^{N+1}_{+})\setminus\{0\}: L'_{\alpha_{0}, V(x_{0}+z_{0})}(u)=0\right\}
$$
and
$$
L_{\alpha_{0}, V(x_{0}+z_{0})}(u):=\frac{\alpha_{0}}{2} \|u\|^{2}_{X^{s}(\R^{N+1}_{+})}+\frac{V(x_{0}+z_{0})}{2} \int_{\R^{N}} u^{2}(x,0)\, dx-\int_{\R^{N}} F(u(x,0))\, dx.
$$
We observe that, by the results in \cite{Aade}, it turns out that $d_{V(x_{0}+z_{0})}>0$.
Then, for $R>0$ large enough we get
\begin{align*}
\liminf_{\e\ri 0} \frac{s}{N} \alpha_{0} \iint_{B_{R}^{+}(z_{\e}+(\frac{x_{\e}}{\e}), 0)} y^{1-2s} |\nabla w_{\e}|^{2}\, dx dy&=\liminf_{\e\ri 0} \frac{s}{N} \alpha_{0} \iint_{B_{R}^{+}(z_{\e}+(\frac{x_{\e}}{\e}), 0)} y^{1-2s} |\nabla w_{\e,1}|^{2}\, dx dy \nonumber \\
&=\liminf_{\e\ri 0} \frac{s}{N} \alpha_{0} \iint_{B_{R}^{+}(0, 0)} y^{1-2s} |\nabla \bar{w}_{\e}|^{2}\, dx dy \nonumber \\
&\geq \frac{s}{N} \alpha_{0} \iint_{B_{R}^{+}(0, 0)} y^{1-2s} |\nabla \bar{w}|^{2}\, dx dy \nonumber \\
&\geq  \frac{1}{2} \frac{s}{N} \alpha_{0}  \iint_{\R^{N+1}_{+}} y^{1-2s} |\nabla \bar{w}|^{2}\, dx dy \nonumber \\
&\geq \frac{1}{2} d_{V(x_{0}+z_{0})}>0.
\end{align*}
On the other hand, arguing as in \eqref{3.1}, it follows from \eqref{3.2} and $|z_{\e}|\ri \infty$ that
\begin{align*}
&\alpha_{0} \iint_{B_{R}^{+}(z_{\e}+(\frac{x_{\e}}{\e}), 0)} y^{1-2s} |\nabla w_{\e}|^{2}\, dx dy \\
&\leq C \iint_{B_{R}^{+}(z_{\e}+(\frac{x_{\e}}{\e}), 0)} y^{1-2s} \left|\nabla \left(\phi_{0}\left(\frac{\e}{\beta} \sqrt{\left|x-\frac{x_{\e}}{\e}\right|^{2}+y^{2}}\right) W_{0}\left(x-\frac{x_{\e}}{\e}, y\right)\right)\right|^{2}\, dx dy+Cd_{0} \\
&\leq C \iint_{B_{R}^{+}(z_{\e},0)} y^{1-2s} |\nabla W_{0}|^{2}\, dx dy+C\e^{2} \iint_{B_{R}^{+}(z_{\e},0)} y^{1-2s} |W_{0}|^{2}\, dx dy+Cd_{0} \\
&\leq C \iint_{B_{R}^{+}(z_{\e},0)} y^{1-2s} |\nabla W_{0}|^{2}\, dx dy+C\e^{2}R^{2} \left(\iint_{B_{R}^{+}(z_{\e},0)} y^{1-2s} |W_{0}|^{2\gamma} \, dx dy  \right)^{\frac{1}{2\gamma}} +Cd_{0}\\
&=Cd_{0}+o(1)
\end{align*}
which leads to a contradiction for $d_{0}>0$ small enough.
Consequently, \eqref{3.25} holds true. \\
Then, by $(f_1)$-$(f_2)$ and \eqref{3.25}, we have as $\e\ri 0$
\begin{align}\label{Ff}
\int_{\R^{N}}  F(\tilde{w}_{\e,1}(x,0))\, dx\ri \int_{\R^{N}}  F(\tilde{w}(x,0))\, dx \mbox{ and } \int_{\R^{N}}  f(\tilde{w}_{\e,1}(x,0))\tilde{w}_{\e,1}(x,0)\, dx\ri \int_{\R^{N}}  f(\tilde{w}(x,0))\tilde{w}(x,0)\, dx.
\end{align}
Moreover, we can see that as $\e\ri 0$
\begin{align}\label{TASSONI}
\int_{\R^{N}} g(\e x+x_{\e}, \tilde{w}_{\e}(x,0))\tilde{w}_{\e,1}(x,0)\, dx\ri \int_{\R^{N}} f(\tilde{w}(x,0))\tilde{w}(x,0)\, dx.
\end{align}
Indeed, using $x_{\e}\ri x_{0}\in \M^{\beta}\subset \Lambda$ and the definition of $\tilde{w}_{\e,1}$, for all $x\in \Gamma^{0}_{\frac{2\beta}{\e}+(5j_{\e}+2)k_{\e}}(0)$ we have
\begin{align}\label{VIDAL1}
g(\e x+x_{\e}, \tilde{w}_{\e}(x,0))\tilde{w}_{\e,1}(x,0)=f(\tilde{w}_{\e}(x,0))\tilde{w}_{\e,1}(x,0),
\end{align} 
since $\e x+x_{\e}\in \M^{4\beta}\subset \Lambda$ for all $x\in \Gamma^{0}_{\frac{2\beta}{\e}+(5j_{\e}+2)k_{\e}}(0)$ and $\e>0$ small. Furthermore, as $\e\ri 0$
\begin{align}\label{VIDAL2}
\int_{\R^{N}} f(\tilde{w}_{\e}(x,0))\tilde{w}_{\e,1}(x,0)\, dx= \int_{\R^{N}} f(\tilde{w}_{\e,1}(x,0))\tilde{w}_{\e,1}(x,0)\, dx+o(1),
\end{align}
because $(f_1)$,$(f_2)$ and \eqref{3.25} yield 
\begin{align*}
\limsup_{\e\ri 0} &\left|\int_{\R^{N}} [f(\tilde{w}_{\e}(x,0))-f(\tilde{w}_{\e,1}(x,0))]\tilde{w}_{\e,1}(x,0)\, dx  \right|\\
&=\limsup_{\e\ri 0} \left|\int_{\Gamma^{0}_{\frac{2\beta}{\e}+(5j_{\e}+2)k_{\e}}(0)\setminus \Gamma^{0}_{\frac{2\beta}{\e}+(5j_{\e}+1)k_{\e}}(0)} [f(\tilde{w}_{\e}(x,0))-f(\tilde{w}_{\e,1}(x,0))]\tilde{w}_{\e,1}(x,0)\, dx  \right| \\
&\leq \delta C+C_{\delta} \limsup_{\e\ri 0} |\tilde{w}_{\e,1}(\cdot,0)|_{L^{p+1}(\R^{N}\setminus \Gamma^{0}_{2\beta/\e}(0))} \\
&\leq \delta C+C_{\delta} \left[\limsup_{\e\ri 0} |\tilde{w}_{\e,1}(\cdot, 0)-\tilde{w}(\cdot, 0)|_{p+1}+\limsup_{\e\ri 0} \int_{\R^{N}\setminus \Gamma^{0}_{2\beta/\e}(0)} |\tilde{w}(x,0)|^{p+1}\, dx\right] \\
&=\delta C \quad \forall \delta>0.
\end{align*}
Gathering \eqref{Ff}, \eqref{VIDAL1} and \eqref{VIDAL2} we get \eqref{TASSONI}.\\
Now, we note that, arguing as before, $\tilde{w}$ satisfies
\begin{align}\label{Ptildew}
\left\{
\begin{array}{ll}
-\dive(y^{1-2s} \nabla \tilde{w})=0 &\mbox{ in } \R^{N+1}_{+}, \\
\frac{1}{\alpha_{0}} \frac{\partial \tilde{w}}{\partial \nu^{1-2s}}=-V(x_{0}) \tilde{w}(\cdot,0)+f(\tilde{w}(\cdot,0)) &\mbox{ in } \R^{N},
\end{array}
\right.
\end{align}
with 
$$
\alpha_{0}:=\lim_{\e\ri 0} M(\|w_{\e}\|^{2}_{X^{s}(\R^{N+1}_{+})})=\lim_{\e\ri 0} M(\|w_{\e,1}\|^{2}_{X^{s}(\R^{N+1}_{+})})=\lim_{\e\ri 0} M(\|\tilde{w}_{\e,1}\|^{2}_{X^{s}(\R^{N+1}_{+})}),
$$
where in the second identity we used that $\|w_{\e}-w_{\e,1}\|_{\e}=o(1)$ thanks to $(iv)$ and $\|w_{\e,2}\|_{\e}=o(1)$, 
and in the third one that  $\tilde{w}_{\e,1}(x,y)=w_{\e,1}(x+\frac{x_{\e}}{\e},y)$.\\
Taking into account \eqref{3.8}, \eqref{TASSONI}, \eqref{Ptildew}, $(iv)$ and $\langle J'_{\e}(w_{\e}), w_{\e,1}\rangle=o(1)$, $\|w_{\e,2}\|_{\e}=o(1)$, $\langle Q'_{\e}(w_{\e}), w_{\e,1}\rangle=0$ and $\tilde{w}_{\e,1}(x,y)=w_{\e,1}(x+\frac{x_{\e}}{\e},y)$, we have
\begin{align*}
&\alpha_{0} \iint_{\R^{N+1}_{+}} y^{1-2s} |\nabla \tilde{w}|^{2}\, dx dy+\int_{\R^{N}} V(x_{0})\tilde{w}^{2}(x,0)\, dx\\
&\leq \liminf_{\e\ri 0} \left[M(\|w_{\e}\|^{2}_{X^{s}(\R^{N+1}_{+})}) \iint_{\R^{N+1}_{+}} y^{1-2s} |\nabla \tilde{w}_{\e,1}|^{2}\, dx dy+\int_{\R^{N}} V(\e x+x_{\e})\tilde{w}_{\e,1}^{2}(x,0)\, dx\right]\\
&\leq \limsup_{\e\ri 0} \left[M(\|w_{\e}\|^{2}_{X^{s}(\R^{N+1}_{+})}) \iint_{\R^{N+1}_{+}} y^{1-2s} |\nabla \tilde{w}_{\e,1}|^{2}\, dx dy+\int_{\R^{N}} V(\e x+x_{\e})\tilde{w}_{\e,1}^{2}(x,0)\, dx\right]\\
&=\limsup_{\e\ri 0} \left[M(\|w_{\e}\|^{2}_{\x}) \iint_{\R^{N+1}_{+}} y^{1-2s} \nabla w_{\e} \nabla w_{\e,1}\, dx dy+\int_{\R^{N}} V_{\e}(x) w_{\e}(x,0) w_{\e,1}(x,0)\, dx\right] \\
&=\limsup_{\e\ri 0} \int_{\R^{N}} g_{\e}(x,w_{\e}(x,0))w_{\e,1}(x,0)\, dx\\ 
&=\lim_{\e\ri 0} \int_{\R^{N}} g(\e x+x_{\e}, \tilde{w}_{\e}(x,0))\tilde{w}_{\e,1}(x,0)\, dx\\
&=\int_{\R^{N}} f(\tilde{w}(x,0))\tilde{w}(x,0)\, dx\\
&=\alpha_{0} \iint_{\R^{N+1}_{+}} y^{1-2s} |\nabla \tilde{w}|^{2}\, dx dy+\int_{\R^{N}} V(x_{0})\tilde{w}^{2}(x,0)\, dx
\end{align*} 
which yields
\begin{align}\label{3.28}
\lim_{\e\ri 0} \iint_{\R^{N+1}_{+}} y^{1-2s} |\nabla w_{\e,1}|^{2}\, dx dy=\lim_{\e\ri 0} \iint_{\R^{N+1}_{+}} y^{1-2s} |\nabla \tilde{w}_{\e,1}|^{2}\, dx dy=\iint_{\R^{N+1}_{+}} y^{1-2s} |\nabla \tilde{w}|^{2}\, dx dy
\end{align}
and
\begin{align}\label{3.29}
\lim_{\e\ri 0}  \int_{\R^{N}} V(\e x)w_{\e,1}^{2}(x,0)\, dx=\lim_{\e\ri 0}  \int_{\R^{N}} V(\e x+x_{\e})\tilde{w}_{\e,1}^{2}(x,0)\, dx=\int_{\R^{N}} V(x_{0})\tilde{w}^{2}(x,0)\, dx.
\end{align}
In particular,
$$
\alpha_{0}=M(\|\tilde{w}\|^{2}_{X^{s}(\R^{N+1}_{+})}).
$$
Putting together \eqref{3.6}, \eqref{Ff}, \eqref{3.28}, \eqref{3.29} we deduce that
$$
\liminf_{\e\ri 0} J_{\e}(w_{\e})\geq \liminf_{\e\ri 0} I_{\e}(w_{\e,1})\geq L_{V(x_{0})}(\tilde{w})
$$
which combined with \eqref{3.2} gives
$$
L_{V(x_{0})}(\tilde{w})\leq c_{V_{0}}.
$$
Since $\tilde{w}\neq 0$, 
it follows from \eqref{4.3G} that 
$$
L_{V(x_{0})}(\tilde{w})\geq c_{V(x_{0})}.
$$ 
Then, using the fact that $x_{0}\in \M^{\beta}\subset \Lambda$, the above inequalities and the monotonicity of $m\mapsto c_{m}$ (see Remark \ref{REMARK}), we have that $V(x_{0})=V_{0}$ and thus $x_{0}\in \M$.
At this point, it is clear that there exist $W\in \mathcal{S}_{V_{0}}$ and $z_{0}\in \R^{N}$ such that $\tilde{w}(x,y)=W(x-z_{0}, y)$.\\
On the other hand, observing that
$$
V(x_{0})=V_{0}\leq V(\e x+x_{\e}) \mbox{ on } \Gamma^{0}_{\frac{2\beta}{\e}+(5j_{\e}+2)k_{\e}}(0),
$$
we combine \eqref{3.28} with \eqref{3.29} to infer that $\tilde{w}_{\e,1}\ri \tilde{w}$ in $X^{1,s}(\R^{N+1}_{+})$ as $\e \ri 0$, which implies that
\begin{align*}
\lim_{\e\ri 0} \left\|w_{\e} -\phi_{0}\left(\frac{\e}{\beta}\sqrt{\left|x-\left(\frac{x_{\e}}{\e}+z_{0}\right)\right|^{2}+y^{2}}\right) W\left(x-\left(\frac{x_{\e}}{\e}+z_{0}\right), y\right)\right\|_{\e}=0.
\end{align*}
This ends the proof of lemma.
\end{proof}

\noindent
\begin{cor}\label{cor5.1G}
For any $d\in (0, d_{0})$ there exist constants $\omega>0$ and $\e_{d}>0$ such that $\|J'_{\e}(w)\|_{(X_{\e})^{-1}}\geq \omega$ for $w\in J_{\e}^{d_{\e}}\cap (E_{\e}^{d_{0}}\setminus E^{d}_{\e})$ and $\e\in (0, \e_{d})$. Here $d_{\e}$ is defined as in \eqref{DE}.
\end{cor}
\begin{proof}
Assume by contradiction that there exist $d\in (0, d_{0})$, $(\e_{n})$ and $(w_{n})$ such that
\begin{align*}
\e_{n}\in \left(0, \frac{1}{n}\right),\quad w_{n}\in J_{\e_{n}}^{d_{\e_{n}}}\cap (E_{\e_{n}}^{d_{0}}\setminus E^{d}_{\e_{n}}), \quad \|J'_{\e_{n}}(w_{n})\|_{(X_{\e_{n}})^{-1}}<\frac{1}{n}.
\end{align*}
By Lemma \ref{Klem}, we can find $(z_{n})\subset \R^{N}$, $x_{0}\in \M$ and $W\in \mathcal{S}_{V_{0}}$ such that
\begin{align*}
\lim_{n\ri \infty}|\e_{n} z_{n}- x_{0}|=0 \mbox{ and } \lim_{n\ri \infty} \|w_{n} -\phi_{0}(\e_{n}\sqrt{|x-z_{n}|^{2}+y^{2}}/\beta) W(x-z_{n}, y)\|_{\e_{n}}=0,
\end{align*}
which imply that
$w_{n}\in E_{\e_{n}}^{d}$ for $n$ sufficiently large. This is impossible because $w_{n}\in E_{\e_{n}}^{d_{0}}\setminus E^{d}_{\e_{n}}$.
\end{proof}
\begin{lem}\label{lem5.1G}
Given $\lambda>0$ there exist $\e_{0}>0$ and $d_{0}>0$ small enough such that
\begin{align*}
J_{\e}(w)>c_{V_{0}}-\lambda \quad \mbox{ for all } w\in E_{\e}^{d_{0}} \quad \mbox{ and } \e\in (0, \e_{0}). 
\end{align*}
\end{lem}
\begin{proof}
If $w\in E_{\e}$ then there exist $W\in \mathcal{S}_{V_{0}}$ and $x'\in \M^{\beta}$ such that 
$$
w(x, y)=\phi_{0}(\sqrt{|\e x-x'|^{2}+\e^{2}y^{2}}/\beta) W(x-(x'/\e), y).
$$ 
Using $L_{V_{0}}(W)=c_{V_{0}}$, $(V_2)$ and $G(x, t)\leq F(t)$ we get
\begin{align*}
J_{\e}(w)-c_{V_{0}}&\geq \frac{1}{2}\left[\wh{M}(\|w\|^{2}_{\x})-\wh{M}(\|W\|^{2}_{\x})\right]+\frac{V_{0}}{2}\int_{\R^{N}} (\phi^{2}_{0}(\e|x|/\beta)-1)W^{2}(x,0)\, dx \\
&\quad-\int_{\R^{N}} F(\phi^{2}_{0}(\e|x|/\beta)W(x,0))-F(W(x,0))\, dx
\end{align*} 
independently of $x'\in \M^{\beta}$. Arguing as in the proof of Lemma \ref{lem4.1G}, we can see that there exists $\e_{0}>0$ such that
$$
J_{\e}(w)-c_{V_{0}}>-\frac{\lambda}{2} \quad \mbox{ for all } w\in E_{\e} \quad \mbox{ and } \e\in (0, \e_{0}).
$$
Now, if $v\in E^{d}_{\e}$, then there exists $w\in E_{\e}$ such that $\|w-v\|_{\e}\leq d$. Hence, $v=w+z$ with $\|z\|_{\e}\leq d$. Observing that $Q_{\e}(w)=0$, we have
\begin{align*}
J_{\e}(v)-J_{\e}(w)&\geq \frac{1}{2} [\wh{M}(\|w+z\|^{2}_{\x})-\wh{M}(\|w\|^{2}_{\x})] +\frac{1}{2} \int_{\R^{N}} V_{\e}(x) [(w(x,0)+z(x,0))^{2}-w^{2}(x,0)]\, dx \\
&\quad-\int_{\R^{N}} G_{\e}(x, w(x,0)+z(x,0))-G_{\e}(x, w(x,0))\, dx.
\end{align*}
Since $E_{\e}$ is uniformly bounded for $\e\in (0, \e_{0})$ (see the estimates in the proof of Lemma \ref{lem4.1G}), we obtain that for $\e\in (0, \e_{0})$
$$
|\|w+z\|^{2}_{\e}-\|w\|^{2}_{\e}|\leq \|z\|^{2}_{\e}+2\|w\|_{\e}\|z\|_{\e}\leq d^{2}+Cd\ri 0 \mbox{ as } d\ri 0.
$$
Moreover, noting that $\wh{M}(t_{2})-\wh{M}(t_{1})=\int_{t_{1}}^{t_{2}} M(\tau)\,d\tau$ and $(M5)$ yield
$$
|\wh{M}(\|w+z\|^{2}_{\x})-\wh{M}(\|w\|^{2}_{\x})|\leq M(C)  |\|w+z\|^{2}_{\x}-\|z\|^{2}_{\x}|\ri 0 \mbox{ as } d\ri 0,
$$
we can find $d_{0}>0$ small enough such that
$$
J_{\e}(v)>J_{\e}(w)-\frac{\lambda}{2}>c_{V_{0}}-\lambda \quad \forall v\in E_{\e}^{d_{0}} \quad \forall \e\in (0, \e_{0}).
$$
This ends the proof of lemma.
\end{proof}
By Corollary \ref{cor5.1G} and Lemma \ref{lem5.1G}, we fix $d_{1}\in (0, \frac{d_{0}}{3})$ and corresponding $\omega>0$ and $\e_{0}>0$ such that, for any $\e\in (0, \e_{0})$,
\begin{align*}
&\|J'_{\e}(w)\|_{(X_{\e})^{-1}}\geq \omega \quad \mbox{ for all } w\in J_{\e}^{d_{\e}}\cap (E_{\e}^{d_{0}}\setminus \E^{d_{1}}_{\e})\\
&J_{\e}(w)>\frac{c_{V_{0}}}{2} \quad \mbox{ for all } w\in E_{\e}^{d_{0}}.
\end{align*}
\begin{lem}\label{lem5.2G}
There exists $\alpha>0$ such that 
\begin{align*}
|t-1/t_{0}|\leq \alpha \mbox{ implies that } \gamma_{\e}(t)\in E_{\e}^{d_{1}} \,\mbox{ for all } \e\in (0, \e_{0}),
\end{align*}
where $\gamma_{\e}$ is given by \eqref{4.6G} and $t_{0}$ was chosen in \eqref{4.2G}.
\end{lem}
\begin{proof}
Firstly, we note that there exists $C_{0}>0$ such that
\begin{align*}
\left\| \phi_{0}\left(\frac{\e}{\beta}\sqrt{|x|^{2}+y^{2}}\right) v \right\|_{\e}\leq C_{0}\|v\|_{\X} \quad \forall \e\in (0, \e_{0}) \quad \forall v\in \X.
\end{align*}
Since the map $\psi: [0, t_{0}]\rightarrow \X$ defined as $\psi(t):=W^{*}_{t}$ is continuous, we can find $\sigma>0$ such that $\|W^{*}_{t}-W^{*}\|_{\X}<\frac{d_{1}}{C_{0}}$ whenever $|t-1|\leq \sigma$. Hence, if $|tt_{0}-1|\leq \sigma$, then $|t-\frac{1}{t_{0}}|\leq \frac{\sigma}{t_{0}}=:\alpha$ and this yields
\begin{align*}
\|\gamma_{\e}(t)-W_{\e,1}\|_{\e}=\left\|\phi_{0}\left(\frac{\e}{\beta}\sqrt{|x|^{2}+y^{2}}\right) (W^{*}_{t t_{0}}-W^{*})\right\|_{\e}\leq C_{0} \|  W^{*}_{tt_{0}}-W^{*}\|_{\X}<d_{1}.
\end{align*}
Since $W_{\e,1}\in E_{\e}$ (recall that $0\in \M$ and $W^{*}\in \mathcal{S}_{V_{0}}$), we deduce that $\gamma_{\e}(t)\in E_{\e}^{d_{1}}$.
\end{proof}
\begin{lem}\label{lem5.3G}
For $\alpha$ given in Lemma \ref{lem5.2G} there exist $\rho>0$ and $\e_{0}>0$ such that
\begin{align*}
J_{\e}(\gamma_{\e}(t))<c_{V_{0}}-\rho, \quad \mbox{ for any } \e\in (0, \e_{0}) \quad \mbox{ and } |t-1/t_{0}|\geq \alpha.
\end{align*}
\end{lem}
\begin{proof}
By $(M5)$ and \eqref{4.2G}, we know that $t=1$ is a maximum point of $L_{V_{0}}(W^{*}_{t})$ in $[0, t_{0}]$ (see the proof of Lemma \ref{Final}). 
Then, we find $\rho>0$ such that
$$
L_{V_{0}}(W^{*}_{t})<c_{V_{0}}-2\rho \mbox{ for } |t-1|\geq t_{0}\alpha.
$$
On the other hand, by Lemma \ref{lem4.1G}, there exists $\e_{0}>0$ such that
$$
\sup_{t\in [0, t_{0}]} |J_{\e}(W_{\e,t})-L_{V_{0}}(W^{*}_{t})|<\rho \mbox{ for } \e\in (0, \e_{0}).
$$
Consequently, for $|t-1|\geq t_{0}\alpha$ and $\e\in (0, \e_{0})$, we have
\begin{align*}
J_{\e}(W_{\e, t})\leq L_{V_{0}}(W^{*}_{t})+|J_{\e}(W_{\e,t})-L_{V_{0}}(W^{*}_{t})|<c_{V_{0}}-2\rho+\rho=c_{V_{0}}-\rho.
\end{align*}
\end{proof}
In the light of Lemma \ref{lem5.2G} and Lemma \ref{lem5.3G}, we can argue as in the proof of Proposition 5.2 in \cite{Gloss} (see also \cite{BJ, FIJ, HL}), to obtain the following result that we state without giving the details.
\begin{lem}\label{lem3.3He}
There exists $\bar{\e}>0$ such that for all $\e\in (0, \bar{\e}]$ there exists a sequence $(w_{n, \e})\subset J_{\e}^{{d}_{\e}+\e}\cap E_{\e}^{d_{0}}$ such that $J'_{\e}(w_{n, \e})\ri 0$ in $(X_{\e})^{-1}$ as $n\ri \infty$.
\end{lem}

\noindent
Now we are ready to give the proof of the main result of this section.
\begin{proof}[Proof of Theorem \ref{thm1}]
By Lemma \ref{lem3.3He},  there exists $\bar{\e}>0$ such that for all $\e\in (0, \bar{\e}]$ there exists a sequence $(w_{n, \e})\subset J_{\e}^{{d}_{\e}+\e}\cap E_{\e}^{d_{0}}$ such that $J'_{\e}(w_{n, \e})\ri 0$ in $(X_{\e})^{-1}$ as $n\ri \infty$. Since $(w_{n,\e})$ is bounded in $X_{\e}$, up to a subsequence, as $n\ri \infty$, we have
\begin{equation}\label{3.30}
w_{n,\e}\rightharpoonup w_{\e} \mbox{ in } X_{\e},
\end{equation}
and
\begin{equation}\label{3.31}
\lambda_{n,\e}:=\left(\int_{\R^{N}} \chi_{\e}(x) w_{n, \e}^{2}(x,0)\, dx-1\right)_{+}\ri \lambda_{\e}.
\end{equation}
Then, it is easy to verify that
\begin{align}\label{3.32}
\left\{
\begin{array}{ll}
-\dive(y^{1-2s} \nabla w_{\e})=0 &\mbox{ in } \R^{N+1}_{+}, \\
\frac{1}{\alpha_{\e}} \frac{\partial w_{\e}}{\partial \nu^{1-2s}}=-V_{\e} w_{\e}(\cdot,0)-4\lambda_{\e} \chi_{\e} w_{\e}(\cdot,0)+g_{\e}(x, w_{\e}(\cdot,0)) &\mbox{ in } \R^{N},
\end{array}
\right.
\end{align}
where
$$
\alpha_{\e}:=\lim_{n\ri \infty} M(\|w_{n, \e}\|^{2}_{X^{s}(\R^{N+1}_{+})}).
$$
By $(M1)$, $(M4)$ and the boundedness of $(w_{n,\e})$ in $X_{\e}$ we know that 
\begin{equation}\label{BM}
m_{0}\leq \alpha_{\e}\leq C \quad \forall \e\in (0, \bar{\e}].
\end{equation}
Next, we show that $(w_{n, \e})$ is tight in $X^{s}(\R^{N+1}_{+})$ (see definition 3.2.1 in \cite{DMV}).
 To prove this, for all fixed $\e\in (0, \bar{\e}]$, take $R>0$ such that $\Lambda_{\e}\subset \Gamma^{0}_{R}(0)$, and set $\phi_{R}(x,y):=\bar{\phi}(\sqrt{|x|^{2}+y^{2}}/R)$ where $\bar{\phi}\in C^{\infty}(\R_{+})$ is such that $\bar{\phi}=0$ in $[0, 1]$, $\bar{\phi}=1$ in $[2, \infty)$, $0\leq \bar{\phi}\leq 1$ and $|\bar{\phi}'|_{\infty}\leq C$. Since $(\phi_{R}w_{n, \e})$ is bounded in $X_{\e}$ for each $\e\in (0, \bar{\e}]$, we deduce that $\langle J'_{\e}(w_{n, \e}), \phi_{R} w_{n, \e}\rangle\ri 0$ as $n\ri \infty$, and so, by the definition of $g_{\e}$, we get
\begin{align}\label{3.33}
\alpha_{\e} &\iint_{\R^{N+1}_{+}} y^{1-2s} |\nabla w_{n, \e}|^{2}\phi_{R}\, dxdy+\int_{\R^{N}} V_{\e}(x) w_{n, \e}^{2}(x,0)\phi_{R}(x,0)\, dx \nonumber\\
&\leq \frac{1}{2} \int_{\R^{N}} V_{\e}(x)  w_{n, \e}^{2}(x,0)\phi_{R}(x,0)\, dx-\alpha_{\e}\iint_{\R^{N+1}_{+}} y^{1-2s} w_{n, \e} \nabla w_{n, \e}\nabla \phi_{R}\, dxdy.
\end{align}
Arguing as in \eqref{3.1}, and using H\"older's inequality, \eqref{BM}, \eqref{3.30} and Lemma \ref{lem2.1}-$(ii)$, we get 
\begin{align}\label{3.34}
&\limsup_{n\ri \infty} \left| \alpha_{\e}\iint_{\R^{N+1}_{+}} y^{1-2s} w_{n, \e} \nabla w_{n, \e}\nabla \phi_{R}\, dxdy \right| \nonumber\\
&\leq \frac{C}{R} \limsup_{n\ri \infty}  \left[\left(\iint_{\R^{N+1}_{+}} y^{1-2s} |\nabla w_{n, \e}|^{2}\, dxdy  \right)^{\frac{1}{2}} \left(\iint_{B_{2R}^{+}(0,0)\setminus B_{R}^{+}(0, 0)} y^{1-2s} |w_{n, \e}|^{2}\, dxdy  \right)^{\frac{1}{2}}  \right] \nonumber \\
&\leq \frac{C}{R} \left(\iint_{B_{2R}^{+}(0,0)\setminus B_{R}^{+}(0, 0)} y^{1-2s} |w_{\e}|^{2}\, dxdy  \right)^{\frac{1}{2}}  \nonumber \\
&\leq C \left(\iint_{B_{2R}^{+}(0,0)\setminus B_{R}^{+}(0, 0)} y^{1-2s} |w_{\e}|^{2\gamma}\, dxdy  \right)^{\frac{1}{2\gamma}}\ri 0 \mbox{ as } R\ri \infty. 
\end{align}
Putting together \eqref{BM}, \eqref{3.33} and \eqref{3.34} we obtain
\begin{align}\label{tight}
\lim_{R\ri \infty} \limsup_{n\ri \infty} \iint_{\R^{N+1}_{+}\setminus B_{2R}^{+}(0, 0)} y^{1-2s} |\nabla w_{n, \e}|^{2}\, dxdy+\int_{\R^{N}\setminus \Gamma_{2R}^{0}(0)} V_{\e}(x) w_{n, \e}^{2}(x,0)\, dx=0,
\end{align}
which implies that $(w_{n, \e})$ is tight in $X_{\e}$. In particular, by \eqref{tight} and the compactness of $H^{s}(\R^{N})\subset L^{2}_{loc}(\R^{N})$, we deduce that $w_{n, \e}(\cdot, 0)\ri w_{\e}(\cdot, 0)$ in $L^{2}(\R^{N})$ as $n\ri \infty$. Hence, by interpolation, $w_{n, \e}(\cdot, 0)\ri w_{\e}(\cdot, 0)$ in $L^{q}(\R^{N})$ for all $q\in [2, \2)$. By the definition of $g_{\e}$, $(f_1)$-$(f_2)$, we have as $n\ri \infty$
\begin{align}\label{3.38}
\int_{\R^{N}} g_{\e}(x, w_{n, \e}(x,0))w_{n, \e}(x,0)\, dx\ri \int_{\R^{N}} g_{\e}(x, w_{\e}(x,0))w_{\e}(x,0)\, dx.
\end{align}
In the light of \eqref{3.30}, \eqref{3.32}, \eqref{3.38}, $\langle J'_{\e}(w_{n, \e}), w_{n, \e}\rangle\ri 0$ and arguing as at the end of the proof of Lemma \ref{Klem}, we deduce that 
\begin{align}\label{LAMBDA}
w_{n, \e}\ri w_{\e} \mbox{ in } X_{\e} \mbox{ as } n\ri \infty, \, \alpha_{\e}=M(\|w_{\e}\|^{2}_{X^{s}(\R^{N+1}_{+})}) \mbox{ and } \lambda_{\e}=\left(\int_{\R^{N}} \chi_{\e}(x) w_{\e}^{2}(x,0)\, dx-1\right)_{+}.
\end{align}
Since $\mathcal{S}_{V_{0}}$ is compact in $X^{1,s}(\R^{N+1}_{+})$, it is easy to check that $0\notin E^{d_{0}}_{\e}$ for $\e>0$, $d_{0}>0$ small. Hence, $w_{\e}\in E_{\e}^{d_{0}}\cap J^{d_{\e}+\e}$ is a nontrivial solution to \eqref{3.32}.

Now, for any sequence $(\e_{n})$ such that $\e_{n}\ri 0$ as $n\ri \infty$, by Lemma \ref{Klem} there exist, up to a subsequence, $(z_{n})\subset \R^{N}$, $x_{0}\in \M$ and $W\in \mathcal{S}_{V_{0}}$ such that
\begin{align}\label{3.39}
\lim_{n\ri \infty}|\e_{n}z_{n}-x_{0}|=0
\end{align}
and
\begin{align*}
\lim_{n\ri \infty}  \|w_{\e_{n}} -\phi_{0}(\e_{n}\sqrt{|x-z_{n}|^{2}+y^{2}}/\beta) W(x-z_{n}, y)\|_{\e_{n}}=0,
\end{align*}
which implies that
\begin{align}\label{FANTA}
\lim_{n\ri \infty}  \|\bar{w}_{\e_{n}} -W\|_{X^{1,s}(\R^{N+1}_{+})}=0,
\end{align}
where $\bar{w}_{\e_{n}}(x, y):=w_{\e_{n}}(x+z_{n}, y)$. 
In view of \eqref{3.32}, \eqref{BM}, \eqref{LAMBDA} and \eqref{FANTA}, we can use a Moser iteration scheme (see for instance \cite{Aampa, AI2, DMV}) and repeat the same arguments in \cite{AM, Aasy, AI2, HZm} to deduce that
\begin{align}\label{3.40}
\lim_{|x|\ri \infty} \bar{w}_{\e_{n}}(x,0)=0 \mbox{ uniformly for } \e_{n} \mbox{ small, }
\end{align}
which guarantees the existence of a constant $\rho>0$ such that $f(\tilde{w}_{\e_{n}}(x,0))\leq \frac{V_{0}}{2} \tilde{w}_{\e_{n}}(x,0)$ for all $|x|\geq \rho$ and $\e_{n}$ small. When $|x|\leq \rho$, it follows from \eqref{3.39} that $\Gamma^{0}_{\e_{n}\rho}(\e_{n}z_{n})\subset \Lambda$ for $\e_{n}$ small enough, and so
\begin{align}\label{3.41}
g_{\e_{n}}(x+z_{n}, \bar{w}_{\e_{n}}(x,0))=f(\bar{w}_{\e_{n}}(x,0)) \mbox{ for } \e_{n} \mbox{ small. }
\end{align}
From \eqref{3.40} and $(f_1)$, we can find $R>0$ big enough such that 
$$
f(\bar{w}_{\e_{n}}(x,0))\leq \frac{1}{2} V(\e_{n}x+\e_{n}z_{n}) \bar{w}_{\e_{n}}(x,0) \mbox{ for } x\in \R^{N}\setminus \Gamma^{0}_{R}(0).
$$
On the other hand, arguing as in \cite{AM, Armi, Aasy}, we see that 
$$
|\bar{w}_{\e_{n}}(x,0)|\leq \frac{C}{1+|x|^{N+2s}} \mbox{ for } \e_{n} \mbox{ small, }
$$
for some $C>0$ independent of $\e_{n}$. Then, noting that $\R^{N}\setminus (\Lambda_{\e_{n}}-z_{n})\subset \R^{N}\setminus \Gamma^{0}_{\frac{\beta}{\e_{n}}}(0)$, we obtain
\begin{align*}
\e_{n}^{-1} \int_{\R^{N}\setminus \Lambda_{\e_{n}}} w^{2}_{\e_{n}}(x,0)\, dx&=\e_{n}^{-1} \int_{\R^{N}\setminus (\Lambda_{\e_{n}}-z_{n})} \bar{w}^{2}_{\e_{n}}(x,0)\, dx \\
&\leq C\e_{n}^{-1} \int_{\R^{N}\setminus \Gamma^{0}_{\frac{\beta}{\e_{n}}}(0)} \frac{1}{(1+|x|^{N+2s})^{2}}\, dx\ri 0 \mbox{ as } n\ri \infty,
\end{align*}
which implies that $Q_{\e_{n}}(w_{\e_{n}})=0$ for $\e_{n}$ small enough. This together with \eqref{3.41} implies that $w_{\e_{n}}$ is a solution to \eqref{EP}. Hence, $u_{\e_{n}}(x):=w_{\e_{n}}(\frac{x}{\e_{n}}, 0)$ is a solution to \eqref{P}. Since $u_{\e}\in L^{\infty}(\R^{N})$, $u_{\e}\geq 0$ in $\R^{N}$, $V$ and $f$ are continuous functions, and using $(M1)$, from the Harnack inequality \cite{CSire, JLX} we have that $u_{\e}>0$ in $\R^{N}$.

Now, let $P_{n}$ be a global maximum point of $\bar{w}_{\e_{n}}(\cdot, 0)$. 
Since $\bar{w}_{\e_{n}}$ solves \eqref{MEP} with $V_{\e_{n}}$ replaced by $V_{\e_{n}}(\cdot+z_{n})$, it follows from $(V_1)$, $(f_1)$-$(f_2)$ that
$$
V_{1} |\bar{w}_{\e_{n}}(\cdot, 0)|_{2}^{2}\leq \frac{V_1}{2}  |\bar{w}_{\e_{n}}(\cdot, 0)|_{2}^{2}+C|\bar{w}_{\e_{n}}(\cdot, 0)|_{\infty}^{\2-2} |\bar{w}_{\e_{n}}(\cdot,0)|_{2}^{2}
$$
which implies that  $|\bar{w}_{\e_{n}}(\cdot, 0)|_{\infty}\geq \delta>0$ for all $n\in \mathbb{N}$.
Then, $\bar{w}_{\e_{n}}(P_{n}, 0)\geq \delta>0$ for all $n\in \mathbb{N}$, and $(P_{n})$ is bounded by \eqref{3.40}. Noting that $u_{\e_{n}}(x)=\bar{w}_{\e_{n}}(\frac{x}{\e_{n}}-z_{n}, 0)$, we deduce that $x_{n}:=\e_{n}P_{n}+\e_{n}z_{n}$ is a global maximum point of $u_{\e_{n}}$. From \eqref{3.39} we get $x_{n}\ri x_{0}\in \M$ as $n\ri \infty$. Finally, we can argue as in \cite{Armi, Aasy, HZm} to deduce the polynomial decay of $u_{\e}$.

\end{proof}

\section{Proof of Theorem \ref{thm2}}
This section is devoted to the proof of Theorem \ref{thm2}. We borrow some arguments used in \cite{ZCDO}.\\
In view of Proposition \ref{prop2.1ZCDO} there exists $\kappa>0$ such that
\begin{equation}\label{3.1ZCDO}
\sup_{u\in \mathcal{S}_{V_{0}}} |u(\cdot, 0)|_{\infty}=\sup_{u\in \widetilde{\mathcal{S}}_{V_{0}}} |u(\cdot, 0)|_{\infty}<\kappa.
\end{equation}
For any $k>\max_{t\in [0, \kappa]} f(t)$, define $f_{k}(t):=\min\{f(t), k\}$. Now, we consider the truncated problem
\begin{equation}\label{TTP}
\left\{
\begin{array}{ll}
\e^{2s}M(\e^{2s-N}[u]^{2}_{s})(-\Delta)^{s}u +V(x) u= f_{k}(u) &\mbox{ in } \R^{N}, \\
u\in H^{s}(\R^{N}), \quad u>0 &\mbox{ in } \R^{N}.
\end{array}
\right.
\end{equation} 
In what follows, we prove that, for small $\e>0$, there exists a positive solution $v_{\e}$ to \eqref{TTP} satisfying the properties of Theorem \ref{thm2}. Clearly, $v_{\e}$ is a solution to \eqref{P} if $|v_{\e}|_{\infty}<\kappa$.
We consider the limiting problem
\begin{equation}\label{LTP}
\left\{
\begin{array}{ll}
M([u]^{2}_{s})(-\Delta)^{s}u +V_{0} u= f_{k}(u) &\mbox{ in } \R^{N}, \\
u\in H^{s}(\R^{N}), \quad u>0 &\mbox{ in } \R^{N},
\end{array}
\right.
\end{equation}
and the corresponding extended problem
\begin{align}\label{ELP}
\left\{
\begin{array}{ll}
-\dive(y^{1-2s} \nabla w)=0 &\mbox{ in } \R^{N+1}_{+}, \\
\frac{1}{M(\|w\|^{2}_{\x})}  \frac{\partial w}{\partial \nu^{1-2s}}=-V_{0} w(\cdot, 0)+f_{k}(w(\cdot, 0)) &\mbox{ in } \R^{N}, 
\end{array}
\right.
\end{align}
whose associated energy functional is given by
$$
L_{V_{0}}^{k}(u)=\frac{1}{2}\wh{M}(\|u\|^{2}_{\x})+\frac{V_{0}}{2} |u(\cdot,0)|^{2}_{2}-\int_{\R^{N}} F_{k}(u(x,0))\, dx.
$$
\begin{lem}\label{lem3.1ZCDO}
Under the same assumptions of Theorem \ref{thm2}, \eqref{ELP} admits a positive ground state solution.
\end{lem}
\begin{proof}
Firstly we show that $f_{k}$ satisfies $(f_1)$-$(f_3)$. It is clear that $(f_1)$-$(f_2)$ are true. Now, for any $u\in \widetilde{\mathcal{S}}_{V_{0}}$, we know that $u$ fulfills the Pohozaev identity
$$
\frac{N-2s}{N} \|u\|^{2}_{\x}=N\int_{\R^{N}} F(u(x,0))-\frac{V_{0}}{2} u^{2}(x,0)\, dx,
$$
which yields 
$$
\int_{\R^{N}} F(u(x,0))-\frac{V_{0}}{2} u^{2}(x,0)\, dx\geq 0.
$$
If $F(u(x,0))-\frac{V_{0}}{2} u^{2}(x,0)\leq 0$ for all $x\in \R^{N}$, then $\frac{F(u(x,0))}{u^{2}(x,0)}=V_{0}>0$ for all $x\in \R^{N}$.
Using $(f'_1)$ and that $u(x,0)\ri 0$ as $|x|\ri \infty$, we get $\frac{F(u(x,0))}{u^{2}(x,0)}\ri 0$ as $|x|\ri \infty$, that is a contradiction.
Then, we can find $x_{0}\in \R^{N}$ such that $F(u(x_{0}, 0))>\frac{V_{0}}{2}u^{2}(x_{0}, 0)$. Since $|u(x_{0}, 0)|<\kappa$, it follows that $F_{k}(u(x, 0))=F(u(x, 0))$ for all $x\in \R^{N}$. Hence, letting $T=u(x_{0}, 0)>0$, we obtain that $F_{k}(T)>\frac{V_{0}}{2}T^{2}$, that is $(f_3)$ is satisfied. From \cite{Aade, BKS, ZDOS} we know that
$$
(-\Delta)^{s}u+V_{0}u=f_{k}(u) \mbox{ in } \R^{N}
$$
admits a radially symmetric ground state solution. At this point, we apply Lemma \ref{lem2.16FIJ} to deduce the assertion.
\end{proof}

Let $\mathcal{S}^{k}_{V_{0}}$ be the set of ground state solutions $u$ to \eqref{LTP} such that $u(0, 0)=\max_{x\in \R^{N}} u(x, 0)$.
Then, by Lemma \ref{lem3.1ZCDO} we deduce that $\mathcal{S}^{k}_{V_{0}}\neq \emptyset$.
\begin{lem}\label{lem3.2ZCDO}
For $k>\max_{t\in [0, \kappa]} f(t)$, we have
$$
\mathcal{S}^{k}_{V_{0}}=\mathcal{S}_{V_{0}}.
$$
\end{lem}
\begin{proof}
In the light of Lemma \ref{lem2.1ZCDO} and Lemma \ref{lem2.2ZCDO} it is enough to prove that $\widetilde{S}^{k}_{V_{0}}=\widetilde{S}_{V_{0}}$. This is proved in Corollary 4.3 in \cite{JLZ}.
\end{proof}

\noindent
Now we provide the proof of the main result of this section.
\begin{proof}[Proof of Theorem \ref{thm2}]
Since $f_{k}$ satisfies $(f_1)$-$(f_3)$, we can invoke Theorem \ref{thm1} to deduce that, fixed $k>\max_{t\in [0, \kappa]} f(t)$, there exists $\e_{0}>0$ such that \eqref{TTP} admits a positive solution $v_{\e}$ for $\e\in (0, \e_{0})$. Moreover, there exists $U\in \mathcal{S}^{k}_{V_{0}}$ and a maximum point $x_{\e}$ of $v_{\e}$ such that $\lim_{\e\ri 0} dist(x_{\e}, \M)=0$ and $v_{\e}(\e\cdot+x_{\e})\ri U(\cdot+z_{0})$ as $\e\ri 0$ in $\h$, for some $z_{0}\in \R^{N}$. Letting $w_{\e}=v_{\e}(\e\cdot+x_{\e})$ we see that $w_{\e}$ satisfies
$$
M(\|w_{\e}\|^{2}_{\x})(-\Delta)^{s}w_{\e}+V_{\e}\left(x+\frac{x_{\e}}{\e}\right)w_{\e}=f(w_{\e}) \mbox{ in } \R^{N}.
$$
Clearly,
$$
m_{0}\leq \inf_{\e<\e_{0}} M(\|w_{\e}\|^{2}_{\x})\leq \sup_{\e<\e_{0}} M(\|w_{\e}\|^{2}_{\x})<\infty.
$$
Then, we can argue as in Step 2 of the proof of Theorem 1.1 in \cite{JLZ} and use Lemma \ref{lem3.2ZCDO} to infer that there exists $\e^{*}>0$ such that $|v_{\e}|_{\infty}<\kappa$ for all $\e\in (0, \e^{*})$, which implies that $f_{k}(v_{\e})=f(v_{\e})$ in $\R^{N}$. In conclusion, $v_{\e}$ is a positive solution to \eqref{P}.
\end{proof}


\end{document}